\documentclass[11pt]{amsart}

\usepackage{graphicx}
\usepackage{amssymb}
\usepackage[all]{xy}

 \newtheorem{theorem}{Theorem}[section]
 \newtheorem{proposition}[theorem]{Proposition}
 \newtheorem{corollary}[theorem]{Corollary}
 \newtheorem{remark}[theorem]{Remark}

 \newtheorem{lemma}[theorem]{Lemma}
 \newtheorem{example}[theorem]{Example}

 \newcommand{\Ker}{\mathop{\rm Ker}\nolimits}
 \newcommand{\im}{\mathop{\rm Im}\nolimits}
 \newcommand{\diam}{\mathop{\rm diam}\nolimits}
 \newcommand{\hX}{{\widehat{X}}}

 \newcommand{\s}{{\mathbb{S}}}
 \newcommand{\NN}{{\mathbb{N}}}
 \newcommand{\ZZ}{{\mathbb{Z}}}
 \newcommand{\RR}{{\mathbb{R}}}
 \newcommand{\St}{{\mathrm{Stab}}}

\begin{document}

\title{GENERAL THEORY OF LIFTING SPACES}  


\author[Gregory R. Conner]{Gregory R. Conner$^*$} 
\address[Gregory R. Conner]{\newline\hspace*{3mm} Math Department, Brigham Young University, Provo, 
UT 84602, USA}
\email{conner@mathematics.byu.edu}
\thanks{$^{*}$ Supported by Simons Foundation collaboration grant 246221.}

\author[Petar Pave\v{s}i\'c]{Petar Pave\v{s}i\'c$^{**}$}
\address[Petar Pave\v{s}i\'c]{\newline\hspace*{3mm} Faculty of Mathematics and Physics, University of Ljubljana, Slovenija}
\email{petar.pavesic@fmf.uni-lj.si}
\thanks{$^{**}$ Supported by the Slovenian Research Agency program P1-0292 and grants 
N1-0083, N1-0064.}


\begin{abstract}
In his classical textbook on algebraic topology Edwin Spa\-ni\-er developed the theory of covering 
spaces within a more general framework of lifting spaces (i.e., Hurewicz fibrations with unique path-lifting 
property). Among other, Spanier proved that for every space $X$ there exists a universal lifting space,
which however need not be simply connected, unless the base space $X$ is semi-locally simply connected. 
The question on what exactly is the fundamental group of the universal space was left unanswered. 

The main source of lifting spaces are inverse limits of covering spaces over $X$, 
or more generally, over some inverse system of spaces converging to $X$. Every metric space $X$ can be  
obtained as a limit of an inverse system of polyhedra, and so inverse limits of covering spaces 
over the system yield lifting spaces over $X$. They are related to the geometry (in particular the fundamental
group) of $X$ in a similar way as the covering spaces over polyhedra are related to the fundamental group of 
their base. Thus lifting spaces appear as a natural replacement for the concept of 
covering spaces over base spaces with bad local properties.

In this paper we develop a general theory of lifting spaces and prove that they are preserved by products, 
inverse limits and other important constructions. We show that maps from $X$ to polyhedra give rise to 
coverings over $X$ and use that to  prove that for a connected, locally path 
connected and paracompact $X$, the fundamental group of the above-mentioned Spanier's universal space is
precisely the intersection of all Spanier groups associated to open covers 
of $X$, and that the later coincides with the shape kernel of $X$. 

We examine in more detail lifting spaces over $X$ that arise as inverse limits of coverings 
over some approximations of $X$. We construct an exact sequence relating the fundamental group of $X$ with 
the fundamental group and the set of path-components of the lifting space, study relation between the group 
of deck transformations of the lifting space with the fundamental group of $X$ and prove an existence 
theorem for lifts of maps into inverse limits of covering spaces. 

In the final section we consider lifting spaces over non-locally path connected base and relate them to the
fibration properties of the so called hat space (or 'Peanification') construction.
\ \\[3mm]
\end{abstract}

\maketitle


\section{Introduction}

The theory of covering spaces is best suited for spaces with nice local properties. In particular, if $X$ 
is a semi-locally simply-connected space, then there is a complete correspondence between coverings 
of $X$ and the subgroups of its fundamental group $\pi_1(X)$ (see \cite[Section II, 5]{Spanier}). 
On the other hand, Shelah \cite{Shelah} showed that if the fundamental group of a Peano continuum $X$ is not
countable, then it is has at least one `bad' point, around which it is not semi-locally simply-connected. 
It then follows that most subgroups
of $\pi_1(X)$ do not correspond to a covering space. Thus, the nice relation between covering spaces and 
$\pi_1(X)$ breaks apart as soon as $X$ is not locally nice. 
Note that the situation considered by Shelah is by no means exotic: Peano continua with uncountable 
fundamental group appear naturally as attractors of dynamical systems \cite{Hata85}, as fractal spaces 
\cite{Mandelbrot85,Massopust89}, as boundaries of non-positively curved groups \cite{KapovichKleiner00}, 
and in many other important contexts.

Attempts to extend covering space theory to more general spaces include Fox overlays \cite{Fox}, 
and more recently, 
universal path-spaces by Fisher and Zastrow \cite{Fisher-Zastrow} (see also \cite{FRVZ}), and Peano 
covering maps by Brodskiy, Dydak, Labuz and Mitra \cite{BDLM}.  
Presently there is still an animated debate concerning the correct way 
to define generalized covering spaces, fundamental groups and other related concepts in order to retain 
as much of the original theory as possible.

In \cite[Chapter II]{Spanier} Spanier introduced coverings spaces and maps
as a special case of a more general concept of fibrations with unique path-lifting property. The 
latter turn out to be much more 
flexible than coverings. In particular, one can always construct the universal fibration over a given base
space $X$,
where universality is interpreted as being the initial object in the corresponding category. This universal
object however need not be simply connected,
and the basic question of what is its fundamental group is left open in Spanier's book. 

As a part of his approach Spanier characterized subgroups of $\pi_1(X)$ that give rise to covering spaces by showing 
that a covering subgroup of $\pi_1(X)$ must contain all $\mathcal{U}$-small loops with respect to some 
cover $\mathcal{U}$ of $X$. Consequently elements of $\pi_1(X)$ that are small with respect to all covers 
of $X$ cannot be 'unravelled' in any covering space and even in any fibration with unique path-lifting
property. This result was one of the motivations for our work.

In this paper we give a systematic treatment of lifting spaces and their properties, with particular emphasis 
to inverse limits of covering spaces, universal lifting spaces and their fundamental groups. In Section 2 we
give an alternative definition of lifting spaces and derive their basic properties, which include stability
under arbitrary products, compositions and inverse limits. Then we study the structure of the category of
lifting spaces over a given base $X$ and prove the existence of the universal lifting space
over $X$. In Section 3 we study subgroups of the fundamental group $\pi_1(X)$ that correspond to covering
spaces over $X$ (without the assumption that $X$ is semi-locally simply-connected). The main result is 
Theorem \ref{thmshapegroup} in which we prove that fundamental group of the universal lifting space over $X$ 
is precisely the shape kernel of $X$ (i.e. the kernel of the homomorphism from the fundamental group of $X$ to 
its shape fundamental group). In Section 4 we restrict our attention to a special class of lifting
spaces over $X$ that can be constructed as limits of inverse systems of covering spaces over polyhedral 
approximations of $X$. We show that much of the theory 
of covering spaces can be extended to this more general class. We construct an exact sequence relating
the fundamental group of the base with the fundamental group and the set of path-components of the lifting
space. Furthermore, we relate the group of deck transformations of a lifting space with the fundamental 
group of the base and the density of the path-components of the lifting space. At this point one may expect
that every lifting space over a given base space $X$ can be obtained as 
a limit of an inverse system of covering spaces over $X$ or its approximations. However, a closer look reveals that this a subtle question, in \cite{Conner-Herfort-Pavesic} we constructed several classes of 
lifting spaces that are not inverse limits of coverings, and in \cite{Conner-Herfort-Pavesic2} we proved 
a detection and classification theorem for lifting spaces that are inverse limits of finite coverings. 
On the other hand, we can 
construct a universal lifting spaces for the class of lifting spaces studied in Section 4. This is achieved in Section 5 by taking a polyhedral expansion of the base
space and considering the corresponding inverse system of universal coverings. It turns out that the so obtained universal lifting space is in many 
aspects analogous to the universal covering space, especially when the base space is locally path-connected. 
The first main result is Theorem \ref{thm Peano continua} which summarizes the basic properties of lifting 
spaces whose base is a Peano continuum. In addition we give in Theorem \ref{thm:lifting criterion} a 
criterion for maps between inverse limits of covering spaces. Finally, in Section 6 we partially extend our
results to spaces that are not locally path-connected. To this end we consider  the hat construction
(or `peanification'),  
which to a non-locally path-connected space assigns the 'closest' locally path-connected one. The main result
is that by applying the hat construction on a locally 
compact metric space we obtain a fibration (in fact, a lifting projection) if and only if the hat space is locally compact.

\ \\

\section{Fundamental groups of lifting spaces}

Much of the exposition of covering spaces in \cite[Chapter II]{Spanier} is done in a more general setting of (Hurewicz) fibrations 
with unique path lifting property. To work with Hurewicz fibrations we will use the following standard characterization in terms of lifting functions.
Every map $p\colon L\to X$ induces a map $\overline p\colon L^I\to  X^I\times L$, $\overline p\colon \gamma\mapsto (p\circ\gamma, \gamma(0))$. 
In general $\overline p$ is not surjective, in fact its image is the subspace 
$$X^I\sqcap L:=\{ (\gamma, l)\in X^I\times L \mid p(l)=\gamma(0)\}\subset X^I \times L.$$ A \emph{lifting function} for $p$ 
is a section of $\overline p$, that is, a map $\Gamma\colon X^I\sqcap L\to L^I$ such that $\overline p\circ\Gamma$ is the identity map on $X^I\sqcap L$. 
Then we have the following basic result:

\begin{theorem} (cf. \cite[Theorem 1.1]{Pavesic-Piccinini})
A map $p\colon L\to X$ is a Hurewicz fibration if and only if it admits a continuous lifting function $\Gamma$. 
\end{theorem}

Moreover, unique path-lifting property for $p$ means that for $(\gamma,l),(\gamma',l')\in X^I\sqcap L$ the equality $\Gamma(\gamma,l)=\Gamma(\gamma',l')$ implies 
that $(\gamma,l)=(\gamma',l')$. This condition is clearly  is equivalent to the injectivity of $\overline p$, which leads to the following definition.

A map $p\colon L\to X$ is a \emph{lifting projection} if $\overline p\colon L^I\to X^I\sqcap L$ is a homeomorphism, or equivalently, if the following diagram
$$\xymatrix{
L^I \ar[d]_{p\circ -} \ar[r]^{{\rm ev}_0} & L \ar[d]^p\\
X^I \ar[r]_{{\rm ev}_0} & X}
$$
is a pull-back in the category of topological spaces. Given a path $\gamma\colon I\to X$ and an element $l\in L$ 
with $p(l)=\gamma(0)$ we will denote by $\langle\gamma,l\rangle$
the unique path in $L$ which starts at $l$ and covers $\gamma$ (i.e. $\overline p(\langle\gamma,l\rangle)=(\gamma,l)$).
The \emph{lifting space} is the triple $(L,p,X)$ where $p\colon L\to X$ is a lifting projection. We will occasionally 
abuse the notation and refer to the space $L$ or the map $p\colon L\to X$ itself as a lifting space over $X$.

Clearly, every covering map is a lifting projection. Before giving further examples we list some basic properties 
of lifting spaces (cf. \cite{Spanier}, Section 2.2., short proofs are included here to illustrate the efficiency of the alternative definition).

\begin{proposition}
\label{propliftings}
\begin{enumerate}
\item Arbitrary pull-backs, compositions, products, fibred products and inverse limits of lifting spaces are 
lifting spaces.
\item In a lifting space $p\colon L\to X$, given a path $\gamma\colon I\to X$ the formula 
$f_\gamma\colon l\mapsto \langle\gamma,l\rangle(1)$
determines a homeomorphism $f_\gamma\colon p^{-1}(\gamma(0))\to p^{-1}(\gamma(1))$ between the fibres. In particular, 
if $X$ is path-connected then any two fibres of $p$ are homeomorphic.
\item A fibration $p\colon L\to X$ is a lifting space if, and only if its fibres are totally path-disconnnected 
(i.e. admit only constant paths).
\end{enumerate}
\end{proposition}
\begin{proof}
\begin{enumerate}
\item All claims follow from general facts about pull-backs as we now show.

{\bf Pull-backs:} Let $p\colon L\to X$ be a lifting projection and let $f\colon B\to X$ be any map. Then in the following commutative cube
$$\xymatrix@=4mm{
        &  {B\sqcap L} \ar[rr] \ar[dd] &                  & L\ar[dd]\\
{(B\sqcap L)^I} \ar[rr] \ar[dd] \ar[ru]&      & {L^I} \ar[dd] \ar[ru]\\
        &  B \ar[rr]     &             & X\\
{B^I}\ar[rr]\ar[ru] &  & {X^I} \ar[ru]}
$$
the front, back and right vertical face are pull-backs, which by abstract nonsense implies that the left vertical square is also a pull-back. 
Therefore the pull-back projection $B\sqcap L\to B$ is a lifting projection.

{\bf Compositions:} If $p\colon L\to X$ and $q\colon K\to L$ are lifting projections then in the following
diagram
$$\xymatrix{
K^I \ar[r] \ar[d]& L^I \ar[r] \ar[d] & X^I\ar[d] \\
K \ar[r]_q & L \ar[r]_q & X}
$$
the two inner squares are pull-backs, which implies that the outer square is a pull-back, hence $p\circ q$ is
also lifting projection.

{\bf Products:} If $\{p_i\colon L_i\to X_i\}$ is a family of lifting projections, then the diagram
$$\xymatrix{
{\prod_i L^I_i} \ar[r] \ar[d] &  {\prod_i L_i}\ar[d]^{\prod p_i} \\
{\prod_i X^I_i} \ar[r] & {\prod_i X_i} }$$ 
is a product of pull-back diagrams, hence a pull-back diagram itself. It follows that $\prod p_i$ is a lifting 
projection.

{\bf Fibred products:} The fibred product of a family $\{p_i\colon L_i\to X\}_{i \in\mathcal{I}}$ of 
lifting spaces over $X$ is obtained by pulling back
their product along the diagonal map $X\to X^\mathcal{I}$, hence is a lifting space by the above.

{\bf Inverse limits:} An inverse system of lifting spaces is given by a directed 
set $\mathcal{I}$, two $\mathcal{I}$-indexed inverse systems $\mathbf{L}=(L_i, u_{ij}\colon L_j \to L_i)$
and $\mathbf{X}=(X_i, v_{ij}\colon X_j \to X_i)$, and a morphism of systems 
$\mathbf{p}\colon \mathbf{L}\to\mathbf{X}$, such that $p_i\colon L_i\to X_i$ is 
a lifting projection for all $i\in\mathcal{I}$. In order to prove that the limit map 
$$\lim_{\longleftarrow} p_i\colon \lim_{\longleftarrow} L_i\to \lim_{\longleftarrow} X_i $$
is a lifting projection it is sufficient to observe that we have the natural identifications
$$\left( \lim_{\longleftarrow}L_i\right)^I= \lim_{\longleftarrow} L_i^I,\;\;\;\;\;
\left( \lim_{\longleftarrow}X_i\right)^I\sqcap \left( \lim_{\longleftarrow}L_i\right)=\lim_{\longleftarrow}\left(X_i^I\sqcap L_i\right)\;,$$
and that the projection $\displaystyle\left( \lim_{\longleftarrow}L_i\right)^I\longrightarrow\left( \lim_{\longleftarrow}X_i\right)^I\sqcap 
\left( \lim_{\longleftarrow}L_i\right)$ is 
a homeomorphism because it is the inverse limit of homeomorphisms $\overline p_i\colon L_i^I\stackrel{\approx}{\longrightarrow} X_i^I\sqcap L_i$.

\item Let $\overline\gamma$ denote the inverse path of the path $\gamma$. We claim that the map $f_{\overline\gamma}\colon p^{-1}(y)\to p^{-1}(x),\;\; l\mapsto \langle\overline\gamma,l\rangle(1)$ is the inverse
of $f_\gamma$. Indeed, since $\overline p\bigl(\overline{\langle\gamma,l\rangle}\bigr)=\bigl(\overline\gamma,\langle\gamma,l\rangle(1)\bigr)$ we get the equality
$\overline{\langle\gamma,l\rangle}=\langle\overline\gamma,\langle\gamma,l\rangle(1)\rangle$ and so 
$$f_{\overline\gamma}(f_\gamma(l))=\langle\overline\gamma,\langle\gamma,l\rangle(1)\rangle(1)=\langle\overline\gamma,\langle\gamma,l\rangle(1)\rangle(1)=\overline{\langle\gamma,l\rangle}(1)=
\langle\gamma,l\rangle(0)=l.$$
That $f_\gamma f_{\overline\gamma}$ is also the identity map is proved analogously.

\item Assume $p\colon L\to X$ is a lifting space and let $\gamma$ be a path in $p^{-1}(x)\subset L$. 
Then $\overline p(\gamma)=
({\rm const}_x,\gamma(0))=\overline p({\rm const}_{\gamma(0)})$, hence $\gamma={\rm const}_{\gamma(0)}$.

Conversely, assume that all fibres of $p$ admit only constant paths. Since $p$ is a fibration there is a map 
$\Gamma\colon X^I\sqcap L\to L^I$ such that $\overline p\circ\Gamma={\rm Id}$, and we only need to prove that 
$\gamma=\Gamma(p\circ\gamma, \gamma(0))$ for all $\gamma\colon I\to L$. For $s\in I$ let $\gamma_s$ denote 
the path $\gamma_s(t):=\gamma(st)$, and let $H$ be the standard homotopy starting at $\overline{(p\gamma)_s}\cdot (p\gamma)_s$ and ending
at ${\rm const}_{p\gamma(0)}$. Let moreover $\widetilde H\colon I\times I\to L$ be a lifting of $H$ starting at 
$\widetilde H|_{0\times I}=\overline{\gamma_s}\cdot\Gamma(p\gamma,\gamma(0))_s$. It is easy to check that the restriction of 
$\widetilde H$ to $I\times 0\cup 1\times I\cup I\times 1$ determines a path in the fibre $p^{-1}(p\gamma(t))$ from $\gamma(s)$ to
$\Gamma(p\gamma,\gamma(0))(s)$, so by the assumption $\gamma(s)=\Gamma(p\gamma,\gamma(0))(s)$. 
\end{enumerate}
\end{proof}

We have recently proved in \cite[Theorem 3.2]{Pavesic} that, under very general assumptions, lifting spaces are preserved by the mapping space
construction, which yields a host of examples of lifting spaces that are very far from being coverings. The following examples 
illustrate typical ways how a lifting space can fail to be a covering space.

\begin{example}
\label{ex infinite product}
Let $p\colon \mathbb{R}\to S^1$ be the usual covering of the circle. 
Then the countable product $\displaystyle p^\NN \colon \mathbb{R}^\NN\to (S^1)^\NN$ 
is a lifting space by Proposition \ref{propliftings}, but is not a covering space. In fact, 
the fibre of $p$ is not a discrete space, being an infinite product of $\mathbb{Z}$. Even more drastically, one can easily verify 
that the infinite product of circles is not semi-locally simply connected  at any point, which means that it cannot have at all a simply 
connected covering space. 
\end{example}

\begin{example}
\label{ex dyadic solenoid}
Another basic example is given by the following inverse limit of $2^n$-fold coverings
$$\xymatrix{
S^1 \ar[d]_2 & S^1 \ar[d]_4 \ar[l]_2 & S^1 \ar[d]_8 \ar[l]_2 & \cdots \ar[l] & \mathbb{S}_2 \ar[l] \ar[d]^p\\
S^1  & S^1 \ar@{=}[l] & S^1 \ar@{=}[l] & \cdots \ar@{=}[l] & S^1 \ar@{=}[l] }
$$
which presents the dyadic solenoid $\mathbb{S}_2$ as a lifting space over the circle. By varying the choice of coverings 
we obtain an entire family of non-equivalent lifting spaces over the circle
leading to the following interesting problem: is it possible to classify all lifting spaces over the circle? 
One should keep in mind that this necessarily require the study of non-locally path-connected total spaces. In fact,  
Spanier \cite[Proposition 2.4.10]{Spanier} proved that a lifting space $p\colon L\to X$ 
over a locally path-connected and semi-locally sim\-ply-connected base $X$ is a covering space if, and only if $L$ is locally path-connected. 
\end{example}

\begin{example}
\label{ex hat}
Let us describe a simple but useful construction that sometimes allow to extend results to 
spaces that are not locally path-connected. Given any space $X$ let $\widehat X$ denote 
the set $X$ endowed with the minimal topology that contains all path-components of open sets in 
$X$. Clearly, if $X$ is locally path-connected then $\widehat X=X$, 
but for non locally path-connected spaces we obtain a strictly stronger topology, so the identity map $\iota\colon\widehat X\to X$ is a continuous bijection but not a homeomorphism. For example, if $W$ is the standard Warsaw circle, 
then one can easily check that $\widehat W$ is homeomorphic to the interval $[0,1)$. Note that this
construction was called \emph{Peanification} in \cite{BDLM}, but some care is needed, because in general 
$\widehat X$ is not a Peano space.

The hat-construction is clearly functorial (in fact, together with the projection to the original
space it forms an idempotent augmented functor), so that for every map $f\colon Y\to X$ we obtain a commutative diagram
$$\xymatrix{
{\widehat Y} \ar[r]^{\widehat f} \ar[d] & {\widehat X} \ar[d]\\
Y \ar[r]_f & X
}$$
It follows that every map from a locally path-connected space to $X$ lifts uniquely to a map to $\widehat X$, so in particular, the projection from the hat space admits
unique path liftings. Even more, it is always a Serre fibration, but it is not in general a Hurewicz fibration (and hence not a lifting space). We are going
to study this question in detail in the last section of the paper.
\end{example}

Every covering space and every locally trivial fibration is an open map. In view of the above examples it would be interesting to know 
whether all lifting projections over a locally path-connected base are open maps.

\ \\
In the solenoid example above the total space is not path-connected. Clearly, if $p\colon L\to X$ is a lifting 
space then the restriction of $p$ to any path-component of $L$ is a lifting space, too. 
In order to study the fundamental groups of lifting spaces, we now restrict our attention to based path-connected spaces. 
Let $\mathsf{Lift}_X$ denote the category  whose objects are path-connected open lifting spaces over $X$, and morphisms are fibre-preserving maps between them. All spaces have base-points and all maps
are base-point preserving, but we systematically omit the base-points from the notation. The category $\mathsf{Lift}_X$ 
shares many properties with its full subcategory of covering spaces $\mathsf{Cov}_X$ but is in some aspects more flexible.

\begin{proposition}
\label{propcatlift}
Morphisms in $\mathsf{Lift}_X$ are lifting projections and $\mathsf{Lift}_X$ is an ordered category 
(i.e. there is at most one morphism between any two objects).
\end{proposition}
\begin{proof}
Let $f\colon L\to K$ be a morphism between lifting spaces $p\colon L\to X$ and $q\colon K\to X$. Then we have the natural identification 
$K^I\sqcap L=(X^I\sqcap K)\sqcap L=X^I\sqcap L=L^I$ induced by $f$, therefore $f$ is a lifting projection.

If $f,g\colon L\to K$ are morphisms in $\mathsf{Lift}_X$ then the unique path-lifting property imply that $f$ and $g$ coincide on path-components. Since $f$ and $g$ coincide on the base-point,
and since $L$ is path-connected, we have $f=g$. 
\end{proof}

Clearly the category $\mathsf{Lift}_X$ has equalizers as there are no parallel pairs of distinct maps. 
It also has products: one can easily check that the categorical product of a set of 
lifting spaces $\{L_i\to X\}_{i\in\mathcal{I}}$ is obtained by taking the path-component of 
the fibred product of $\{L_i\to X\}_{i\in\mathcal{I}}$ containing the base-point.  
Since categorical products and equalizers suffice for the construction of any set-indexed categorical limit 
we obtain the following fact.

\begin{proposition}
\label{proplimits}
Category $\mathsf{Lift}_X$ has arbitrary small (i.e. set-indexed) limits. 
\end{proposition}

\begin{corollary} 
\label{coruniv}
For every  path-connected space $X$ the category $\mathsf{Lift}_X$ has the universal (initial)
object $\widetilde X$. The correspondence $X\mapsto \widetilde X$ determines an idempotent augmented functor, 
as $f\colon X\to Y$ induce the commutative diagram
$$\xymatrix{
\widetilde{X}\ar[r]^{\tilde f} \ar[d] & \widetilde{Y}\ar[d]\\
X \ar[r]_f & Y
}$$
\end{corollary}
\begin{proof} 
We first observe that the isomorphism classes of objects in $\mathsf{Lift}_X$ form a set. In fact 
by \cite[Theorem 2.3.9]{Spanier} the points on any fibre of a lifting space $p\colon L\to X$ are 
in bijection with the set cosets of the subgroup $p_\sharp(\pi_1(L))$ in $\pi_1(X)$.
This means that every lifting space over $X$ corresponds to a choice of a subgroup of $\pi_1(X)$, together 
with a choice of a topology on the cartesian product of the set $X$
with the set of cosets of $p_\sharp(\pi_1(L))$ in $\pi_1(X)$. We conclude that the class of possible lifting 
spaces over $X$ whose total spaces is path-connected, forms a set. By Proposition \ref{proplimits} the 
categorical product of a set of representatives of all objects in $\mathsf{Lift}_X$ exists and is clearly
the initial object of the category. The other properties of the universal lifting space follow from
general properties of initial objects.
\end{proof}

As for covering spaces, it is of crucial importance to determine the fundamental group of the universal lifting space.
The following result is a step in that direction. 

\begin{proposition}
\label{propfunprod}
Fundamental group of a categorical product in $\mathsf{Lift}_X$ is the intersection of the fundamental groups 
of its factors. 
\end{proposition}
\begin{proof}
Let $p\colon L\to X$ be the categorical product of the family of lifting spaces $\{p_i\colon L_i\to X\}$.
By Proposition \ref{propcatlift} the projection maps $q_i\colon L\to L_i$ are lifting projections, so 
by \cite[Theorem 2.3.4]{Spanier} they induce monomorphisms $(q_i)_\sharp\colon \pi_1(L)\to\pi_1(L_i)$.
It follows that $\pi_1(L)\cong\im p_\sharp\le\pi_1(X)$ is contained in $\bigcap_i \im (p_i)_\sharp$.
For the converse implication, let the loop $\alpha\colon S^1\to X$ represent an element of 
$\bigcap_i \pi_1(L_i)\cong\bigcap_i \im (p_i)_\sharp\le\pi_1(X)$. 
By the unique path lifting property
there are unique lifts $\alpha_i\colon S^1\to L_i$ for the loop $\alpha$, so they define an element
$\widetilde \alpha\in\pi_1(L)$. We therefore conclude that $\pi_1(L)\cong\bigcap_i\pi_1(L_i)$.
\end{proof}

In order to achieve a more precise identification of the fundamental group of $\widetilde X$ we need a better understanding
of subgroups of the $\pi_1(X)$ that correspond to covering spaces of $X$.

\ \\

\section{Covering subgroups}

In this section we consider the question which subgroups of the fundamental group of $X$ correspond to coverings 
of $X$ and relate them to the shape kernel of $X$. 

Let $(X,x_0)$ be a based space. A subgroup $G\le\pi_1(X,x_0)$ is a \emph{covering subgroup} if there is 
a covering space $p\colon (\widetilde X,\tilde x_0)\to (X,x_0)$, such that  $\im p_\sharp=G$.
Spanier gave a simple characterization of covering subgroups in terms of $\mathcal{U}$-small loops. 
Given a covering $\mathcal{U}$ of $X$ a based loop in $(X,x_0)$ is said to be \emph{$\mathcal{U}$-small} if it is 
of the form 
$\gamma\cdot\alpha\cdot\overline\gamma$ where $\alpha$ is a (non-based) loop whose image is contained in some element of 
$\mathcal{U}$ and $\gamma$ is a path in $X$ connecting $x_0$ and $\alpha(0)$. 
We denote by $\pi_1(X,x_0;\mathcal{U})$
the subgroup of $\pi_1(X,x_0)$ generated by classes of $\mathcal{U}$-small loops. 
It is clear that $\pi_1(X,x_0;\mathcal{U})$
is always a normal subgroup of $\pi_1(X,x_0)$, and that $\pi_1(X,x_0;\mathcal{V})$ is contained in $\pi_1(X,x_0;\mathcal{U})$
whenever $\mathcal{V}$ is a covering of $X$ that refines $\mathcal{U}$. 

The covering subgroups can be characterized as follows:

\begin{theorem}[\cite{Spanier} Lemma 2.5.11 and Theorem 2.5.13]
\label{thm Spanier}
Let $X$ be connec\-ted and locally path-connected. Then $G\le\pi_1(X,x_0)$ is a covering subgroup if,
and only if $G$ contains a subgroup of the form $\pi_1(X,x_0;\mathcal{U})$ for some  cover $\mathcal U$ of $X$.
\end{theorem}

A natural source of covering subgroups are continuous maps into polyhedra (or more generally, into semi-locally simply connected spaces).

\begin{corollary}
\label{cor kernels}
Let $f\colon X\to K$ be a map from $X$ to a semi-locally simply connected space $K$. 
Then the kernel of the induced homomorphism $f_\sharp\colon \pi_1(X,x_0)\to \pi_1(K,f(x_0))$ is a covering subgroup of $\pi_1(X,x_0)$.
\end{corollary}
\begin{proof}
Let $\mathcal{U}$ be a cover of $K$, such that for all $U\in\mathcal{U}$ the inclusion $U\hookrightarrow K$ induces a trivial homomorphism on the 
fundamental group. Then the group $\pi_1(X,x_0;f^{-1}{\mathcal U})$ 
is contained in the kernel of $f_\sharp$  because $f$ maps every $f^{-1}{\mathcal U}$-small loop in $X$ to a homotopically trivial loop in $K$.
Theorem \ref{thm Spanier} then implies that $\Ker f_\sharp$ is a covering subgroup of $\pi_1(X,x_0)$.
\end{proof}

For a partial converse to the above result assume that $X$ has a numerable cover $\mathcal{U}=\{U\}$ 
with a subordinated locally finite partition of unity $\{\rho_U\}$, and let $|\mathcal{U}|$ denote the geometric 
realization of the nerve of $\mathcal{U}$. Then the formula $f(x):=\sum_{U\in\mathcal{U}} \rho_U(x)\cdot U$ defines a map  
$f\colon X\to |\mathcal{U}|$. It is well known that the choice of the partition of unity does not affect
the homotopy class of $f$, so the induced homomorphism $f_\sharp$ depends only on the cover $\mathcal{U}$.

\begin{lemma}
\label{lem ex seq}
Let $X$ be a path-connected space with a numerable cover $\mathcal{U}$ consisting of path-connected open sets. 
Then there is a short exact sequence 
$$1\longrightarrow \pi_1(X,x_0;2\mathcal{U})\longrightarrow\pi_1(X,x_0)\stackrel{f_\sharp}{\longrightarrow}
\pi_1(|\mathcal{U}|, f(x_0))\longrightarrow 1,$$
where $2\mathcal{U}$ is the cover of $X$ consisting of all unions of pairs of intersecting sets in $\mathcal{U}$.
\end{lemma}
\begin{proof}
By a suitable modification of the partition of unity we may assume without loss of generality that every $U\in\mathcal{U}$ contains some 
point $x_U\in U$ such that $\rho_U(x_U)=1$. 

For every intersecting pair of sets $U,V\in\mathcal{U}$ choose a path in $U\cup V$ between $x_U$ and $x_V$. 
These paths determine a map $g\colon |\mathcal{U}^{(1)}|\to X$. For every 1-simplex $\sigma$ 
in $|\mathcal{U}^{(1)}|$ the image $f(g(\sigma))$ is contained in the open star of $\sigma$, which implies that 
the map $f\circ g$ is contiguous to the inclusion $i\colon|\mathcal{U}^{(1)}|\hookrightarrow |\mathcal{U}|$
so the  following diagram homotopy commutes:
$$\xymatrix{
X \ar[rr]^f & & |\mathcal{U}| \\
& |\mathcal{U}^{(1)}| \ar[lu]^g \ar@{^(->}[ru]_i}
$$
By applying the fundamental group functor we obtain the diagram
$$\xymatrix{
\pi_1(X,x_0) \ar[rr]^{f_\sharp} & & \pi_1(|\mathcal{U}|, f(x_0)) \\
& \pi_1(|\mathcal{U}^{(1)}|,f(x_0)) \ar[lu]^{g_\sharp} \ar@{^(->}[ru]_{i_\sharp}}
$$
Since the homomorphism $i_\sharp$ is surjective, so must be $f_\sharp$. 

If $U,V\in\mathcal{U}$ intersect then every loop, whose image is contained in $U\cup V$ is mapped by $f$ 
to the star of the simplex in $|\mathcal{U}|$ spanned by the vertices
$U$ and $V$. It follows that $f_\sharp$ is trivial on $2\mathcal{U}$-small loops, therefore
$\pi_1(X,x_0;2\mathcal{U})\subseteq \Ker f_\sharp$.

For the converse we extend the above diagram to obtain the following one:
$$\xymatrix{
{\Ker i_\sharp} \ar[r] \ar@{-->}[d]_{g_\sharp} & {\pi_1(|\mathcal{U}^{(1)}|,u_0)} \ar[r]^{i_\sharp} \ar[d]_{g_\sharp} & {\pi_1(|\mathcal{U}|,u_0)} \ar@{-->}[d] &
{\!\!\!\!\!\!\!\!\!\!\!\!\!\!\!\!\!\!\!\!\!\!\!\!\!\!\!\!\!\!\!\!\!\!\cong \pi_1(X,x_0)/\Ker f_\sharp}\\
{\pi_1(X,x_0;2\mathcal{U})} \ar[r] & {\pi_1(X,x_0)} \ar[r] \ar[ru]^{f_\sharp}& {\pi_1(X,x_0)/\pi_1(X,x_0;2\mathcal{U})}
}$$
Since $\Ker i_\sharp$ is generated by loops given by the boundaries of 2-simplexes in the nerve of $\mathcal{U}$ we may use the same method as 
in the proof of Theorem 7.3(2) in \cite{Cannon-Conner} to show that every such 
loop is mapped by $g_\sharp$ to a sum of $2\mathcal{U}$-small loops, so $g_\sharp(\Ker i_\sharp)\subseteq\pi_1(X,x_0;2\mathcal{U})$. 
Thus we obtain the induced natural map $\pi_1(X,x_0)/\Ker f_\sharp\to \pi_1(X,x_0)/\pi_1(X,x_0;2\mathcal{U}),$ 
and so $\Ker f_\sharp\subseteq\pi_1(X,x_0;2\mathcal{U})$.
\end{proof}

\begin{lemma}
\label{lem covers}
For every cover $\mathcal{U}$ of a paracompact space $X$ there is a numerable cover $\mathcal{V}$ such that
its double $2\mathcal{V}$ is a refinement of $\mathcal{U}$. Moreover, if $X$ is locally path-connected, then 
$\mathcal{V}$ can be chosen so that its elements are path connected.
\end{lemma}
\begin{proof}
Let $\mathcal{U}'$ be a locally finite refinement of $\mathcal{U}$ with a subordinated partition
of unity. Then we have a map $f\colon X\to |\mathcal{U}'|$ defined as before. Let $\mathcal{V}$ be the cover of $X$ 
obtained by taking preimages of open stars of vertexes in the barycentric subdivision of the nerve of $\mathcal{U}'$.
Clearly, two elements of $\mathcal{V}$ can have a non-empty intersection only if they are both contained in some
element of $\mathcal{U}'$, which means that $2\mathcal{V}$ refines $\mathcal{U}'$, and hence $\mathcal{U}$. 

If $X$ is locally path-connected, then we can further refine $\mathcal{V}$ by taking the cover formed
by the path-components of elements of $\mathcal{V}$.
\end{proof}

Assume that a pointed space $(X,x_0)$ can be represented as a limit of an inverse system of pointed polyhedra 
$(X,x_0)=\varprojlim \big((K_i,k_i),p_{ij},\mathcal{I}\big)$ (or more generally, that $X$ has
a polyhedral resolution in the sense of \cite{Mardesic-Segal}). Then the homomorphisms $(p_i)_\sharp\colon \pi_1(X,x_0)\to \pi_1(K_i,k_i)$ 
induce a homomorphism $\partial\colon\pi_1(X,x_0)\to \check\pi_1(X,x_0)$, where 
$\check\pi_1(X,x_0)$ is the so called \emph{first shape group} of $X$, and is defined as the limit of the inverse system of 
fundamental groups $\big(\pi_1(K_i,k_i),(p_{ij})_\sharp,\mathcal{I}\big)$. Although the definitions are based on a specific resolution of $X$, it is a standard fact
(see \cite{Mardesic-Segal}) that both the first shape group of $X$ and the homomorphism $\partial\colon\pi_1(X,x_0)\to \check\pi_1(X,x_0)$ are independent of the chosen 
resolution for $X$.

We are now ready to prove the main theorem of this section which relates various subgroups of the fundamental group.

\begin{theorem}
\label{thmshapegroup}
Let $X$ be a connected, locally path-connected and paracompact space.
Then the the following subgroups of $\pi_1(X,x_0)$ coincide:
\begin{enumerate}
\item[(1)] Intersection of all groups of $\mathcal{U}$-small loops $\pi_1(X,x_0;\mathcal{U})$ 
for $\mathcal{U}$ a cover of $X$.
\item[(2)] Intersection of all covering subgroups of $X$.
\item[(3)] Intersection of all kernels $\Ker f_\sharp$ for maps $f$ from $X$ to a polyhedron.
\item[(4)] The \emph{shape kernel} of $X$, defined as ${\rm ShKer(X)}:=\Ker(\partial\colon\pi_1(X)\to\check{\pi}_1(X))$.
\end{enumerate}
\end{theorem}
\begin{proof}
The inclusion (1)$\subseteq$(2) follows from Theorem \ref{thm Spanier} because every covering subgroup contains some 
subgroup of $\mathcal{U}$-small loops. Corollary \ref{cor kernels} implies that (2)$\subseteq$(3).

To prove the inclusion (3)$\subseteq$(1) observe that by Lemma \ref{lem ex seq} (3) is contained in the intersection 
of all groups of the  form $\pi_1(X,x_0;2\mathcal{U})$, while by Lemma \ref{lem covers} the latter is contained 
in (4).

Finally  the equality (3)=(4) amounts to the standard description of the shape kernel.
\end{proof}

As a consequence we obtain the following description of the fundamental group of the universal lifting space:

\begin{theorem}
If $X$ is a connected, locally path-connected and paracompact then 
$\pi_1(\widetilde X)$ coincides with the shape kernel of $X$.
\end{theorem}
\begin{proof}
By Corollary \ref{coruniv} and Proposition \ref{propfunprod}  the fundamental group of the universal lifting 
space is contained in the intersection of all covering subgroups of $\pi_1(X)$, which by Theorem \ref{thmshapegroup}
coincides with the shape kernel of $X$. For the converse, take a loop $\alpha$ representing an element of the intersection
$\bigcap\pi_1(X;\mathcal{U})$, and approximate $\alpha$ by a sequence of homotopically trivial loops 
$\alpha_i$, such that $\alpha\simeq\alpha_i(\mathrm{mod}\ \mathcal{U}_i)$. 
As each $\alpha_i$ lift to a loop in the universal space, we may apply the fibration property to show that 
$\alpha$ also lift to a loop, hence $\alpha\in\pi_1(\widetilde X)$. 
\end{proof}


\ \\

\section{Inverse limits of coverings}

Inverse limits of coverings are in many aspects the most tractable class of lifting spaces (with the exception of coverings, of course).
To simplify the notation we agree that all spaces have base-points 
which are omitted from the notation whenever they are not explicitly used,  and all maps are base-point preserving.

Let $\mathcal{I}$ be a directed set and $\mathbf{X}=(X_i,u_{ij}\colon X_j\to X_i,i,j\in\mathcal{I})$ an $\mathcal{I}$-indexed 
inverse system of path-connected and semi-locally simply-connected spaces. For each $i\in \mathcal{I}$ let $q_i\colon\widetilde X_i\to X_i$ be
the universal cover of $X_i$. By the standard lifting criterion for maps between covering spaces (see \cite[Theorem 2.4.5]{Spanier}) there are unique
maps $\tilde u_{i,j}\colon \widetilde X_j\to \widetilde X_i$ such that the diagram 
$$\xymatrix{
\widetilde X_j \ar[r]^{\tilde u_{ij}} \ar[d]_{q_j}  & \widetilde X_i \ar[d]^{q_i}\\
X_j \ar[r]_{u_{ij}} & X_i
}$$
commutes. Furthermore, $\mathbf{\widetilde X}:=\big(\widetilde X_i,\tilde u_{i,j},\mathcal{I}\big)$ is an inverse system of spaces and 
$\mathbf{q}:=(q_i)\colon\mathbf{\widetilde X}\to\mathbf{X}$ is a level-preserving mapping of inverse systems.

More generally, let us for every $i\in\mathcal{I}$ choose a subgroup $G_i\le \pi_1(X_i)$ and consider maps $q_i\colon \widetilde X_i\to
\widetilde X_i/G_i$ and $p_i\colon \widetilde X_i/G_i\to X_i$, where  $p_i$ is the covering projection corresponding to the group $G_i$. 
If $(u_{ij})_\sharp(G_j)\subseteq G_i$, then again by the lifting theorem for covering spaces there exist unique 
maps $\bar u_{ij}\colon \widetilde X_j/G_j\to \widetilde X_i/G_i$ such that the following diagram  commutes 
$$\xymatrix{
\widetilde X_j \ar[r]^{\tilde u_{ij}} \ar[d]_{q_j}  & \widetilde X_i \ar[d]^{q_i}\\
\widetilde X_j/G_j \ar[r]^{\bar u_{ij}} \ar[d]_{p_j}  & \widetilde X_i/G_i \ar[d]^{p_i}\\
X_j \ar[r]_{u_{ij}} & X_i
}$$
We will say that the an $\mathcal{I}$-indexed family of groups $G_i\le \pi_1(X_i)$ form a \emph{coherent thread} with respect to the inverse system $\mathbf{X}$ if 
$(u_{ij})_\sharp(G_j)\subseteq G_i$ for all $i\le j$ or, in other words, if $\mathbf{G}=\big(G_i,(u_{ij})_\sharp,\mathcal{I}\big)$ is an inverse system of groups.

Clearly, every coherent thread $\mathbf{G}$ for $\mathbf{X}$ induces an inverse system of spaces 
$\mathbf{\widetilde X/G}:=(\widetilde X_i/G_i,\bar u_{ij},\mathcal{I})$. Moreover, the inverse systems $\mathbf{X}$, $\mathbf{\widetilde X}$ and 
$\mathbf{\widetilde X/G}$ are related by level preserving morphisms of inverse systems 
$\mathbf{p}:=(p_i)\colon \mathbf{\widetilde X/G}\to \mathbf{X}$ and $\mathbf{q}:=(q_i)\colon \mathbf{\widetilde X}\to \mathbf{X/G}$.
Observe that the systems $\mathbf{\widetilde X}$ and $\mathbf{X}$ may be viewed as special instances of $\mathbf{\widetilde X/G}$ 
with respect to coherent threads consisting of trivial groups or of groups $G_i=\pi_1(X_i)$ respectively.

Since the bonding maps are base-point preserving, the inverse limits $X:=\varprojlim\bf{X}$, $\widetilde X:=\varprojlim\bf{\widetilde X}$ 
and $\widetilde X_\mathbf{G}:=\varprojlim\bf{\widetilde X/G}$ are non-empty. We will denote by $u_i\colon X\to X_i$, $\tilde u_i\colon \widetilde X\to \widetilde X_i$ 
and $\bar u_i\colon \widetilde X_\mathbf{G}\to \widetilde X_i/G_i$ the projections from the limit spaces to the system.

\begin{proposition}
\label{prop open map}
The limit map $p_\mathbf{G}:=\varprojlim\mathbf{p}\colon \widetilde X_\mathbf{G}\to X$ is a lifting projection. 
Moreover, if $X$ is locally path-connected then $p_\mathbf{G}$ is an open map.
\end{proposition}
\begin{proof}
The first claim follows directly from the fact that inverse limits of lifting projections is a lifting projection, as proved in Proposition \ref{propliftings}. 

Toward the proof of the second claim observe that $X$ and $\widetilde X_\mathbf{G}$ may be viewed as subspaces of the topological products
$\prod_{i\in\mathcal I} X_i$ and $\prod_{i\in\mathcal I} X_i/G_i$, respectively. Therefore, it is sufficient to show that $p_\mathbf{G}(\widetilde U)$ is 
open in $X$ for every sub-basic  open set of the form $\widetilde U=X_\mathbf{G}\cap (\widetilde U_i\times \prod_{j\ne i} X_j/G_j)$, where $\widetilde U_i$ is
an open  subset of $X_i/G_i$ that is homeomorphically projected to an elementary open subset $U_i=p_i(\widetilde U_i)\subset X_i$. 
Let $x=(x_i)\in p_\mathbf{G}(\widetilde U)$ be the projection of the point  $\tilde x=(\tilde x_i)\in \widetilde U$. 
Since $X$ is locally path-connected, there exists a path-connected open set $V\subset X$ such that $x\in V\subset X\cap (U_i\times \prod_{j\ne i} X_j)$. 
For every $y\in V$ we can find a path $\alpha\colon (I,0,1)\to (V,x,y)$. Then there is a unique lifting 
$\widetilde\alpha\colon (I,0)\to (\widetilde U,\tilde x)$ of $\alpha$ along the lifting projection $p$. As $p_i\colon\widetilde U_i\approx U_i$
the construction of the lifting function implies that $\widetilde \alpha(1)_i\in \widetilde U_i$, therefore $y=p(\widetilde\alpha(1))\in p(\widetilde U)$.
We may therefore conclude that $V\subset p(\widetilde U)$, and consequently, that $p(\widetilde U)$ is open in $X$.  
\end{proof}

The same argument can be used to prove a more general statement that if $\mathbf{G}$ and $\mathbf{G'}$ are coherent threads such that $G_i\le G'_i\le \pi_1(X_i)$
for every  $i\in\mathcal{I}$, then we obtain a lifting projection $\widetilde X_\mathbf{G}\to \widetilde X_\mathbf{G'}$, which is an open map, whenever
$\widetilde X_\mathbf{G'}$ is locally path-connected.

In the following proposition we give an explicit description of the fibre of $p_\mathbf{G}$ in terms of the fundamental groups of the spaces in the inverse system $\mathbf{X}$
and the coherent thread $\mathbf{G}$:

\begin{proposition}
\label{prop fibre limit}
The fibre of $p_\mathbf{G}$ is naturally homeomorphic to the inverse limit of the system of cosets $\big(\pi_1(X_i)/G_i,(u_{ij})_\sharp,\mathcal{I}\big)$.
\end{proposition}
\begin{proof}
For each $i\in \mathcal{I}$ let $x_i$ and $\bar x_i$ denote respectively the base-points of $X_i$ and $\widetilde X_i/G_i$, and let $x=(x_i)$ be the base-point of $X$. 
The corresponding fibres over $x_i$ and $x$ are  $F_i:=p_i^{-1}(x_i)\subset \widetilde X_i/G_i$ and $F:=p_\mathbf{G}^{-1}(x_i)$.
If $i\le j$ then the restriction of $\bar u_{ij}$ maps $F_j$ to $F_i$ so we obtain the inverse system $\big(F_i,\bar u_{ij},\mathcal{I}\big)$ whose limit
is precisely $F$, and the projection maps  $F\to F_i$ may be identified with the restrictions of the projections $\bar u_i\colon\widetilde X_\mathbf{G}\to \widetilde X_i/G_i$.

It is well-known that the function $\partial\colon\pi_1(X_i)\to F_i$, that to every loop $\alpha\in\pi_1(X_i)$ assigns the end point of the lifting 
of $\alpha$ to $\widetilde X_i$, $\partial(\alpha):=\langle\widetilde \alpha,\bar x_i\rangle(1)\in F_i$, induces a bijection $l_i\colon \pi_1(X_i)/G_i\to F_i$. 
The bijections $l_i$ are compatible with the bonding homomorphisms in the inverse system, as for each $i\le j$ we have a commutative diagram
$$\xymatrix{
\pi_1(X_j)/G_j \ar[r]^{(\bar u_{ij})_\sharp}\ar[d]_{l_j} &\pi_1(X_i)/G_i\ar[d]^{l_i}\\
F_j \ar[r]_{\bar u_{ij}}  & F_i }
$$
Observe that the bijections $l_i$ are actually homeomorphisms, as the fibres of covering spaces are discrete topological spaces. We conclude that the morphisms
of inverse systems $(l_i)$ induces a natural homeomorphism between $F=\varprojlim F_i$ and $\varprojlim \pi_1(X_i)/G_i$.
\end{proof}

We may be lead to expect that $p_\mathbf{G}$ is never a covering projection but that is not the case. 
Indeed, to determine if $p_\mathbf{G}$ is a covering projection it is sufficient to consider the topology on its fibre.
The product topology on the limit of an inverse system of discrete spaces is not discrete, unless almost all bonding maps in the system are injective.
Thus we have the following corollary (where we assume, for simplicity, that $\mathcal{I}=\mathbb{N}$, so that the inverse system is in fact an inverse sequence).

\begin{corollary}
$p_\mathbf{G}$ is a covering projection if and only if the connecting morphisms in the inverse system $\big(\pi_1(X_i)/G_i,(u_{ij})_\sharp,\mathcal{I}\big)$
are eventually injective.
\end{corollary}
\begin{proof}
If there exists $N$ such that $(u_{i\,i-1})_\sharp$ are injective 
for $i>N$ then $p_\mathbf{G}$ is the pullback of the covering projection $p_N\colon \widetilde X_N/G_N\to X_N$  and so it is itself a covering projection.
Conversely, if infinitely many bonding maps in the sequence are non-injective then the limit space is not discrete, hence $p_\mathbf{G}$ is a lifting projection 
but not a covering.
\end{proof} 

Observe that Examples \ref{ex infinite product} (infinite product of coverings) and \ref{ex dyadic solenoid} (dyadic solenoid) are inverse limits of coverings
whose fibres are not discrete, so they are not covering spaces over the circle.
On the other hand, the fibre of the hat construction described in Example \ref{ex hat} is discrete but the total space is usually disconnected, so it is not a covering
space in the usual sense, but rather a disjoint union of coverings. Here is another interesting lifting space:

\begin{example}
\label{ex Warsawonoid}
Let $W$ be the Warsaw circle, and let $(W_i,u_{ij},\NN)$ be the usual sequence of approximations of $W$ by topological annuli. 
Then $\pi_1(W_i)\cong\ZZ$ and $(u_{ij})_\sharp$ are isomorphisms. By choosing a coherent thread $\mathbf{G}$ with 
$G_i:=2^i\ZZ$ we obtain an inverse sequence which is at group level analogous to that of example  \ref{ex dyadic solenoid}. 
The limit $p_\mathbb{G}\colon \widetilde W_\mathbb{G}\to W$ is an interesting lifting space that resembles a dyadic solenoid over $W$, so we
may call it a \emph{dyadic Warsawonoid}.
\end{example}

\subsection{Homotopy exact sequence}

Next we consider the long homotopy exact sequence of the lifting space $p_\mathbf{G}$. As the fibres of a lifting space are totally path-disconnected
(cf.  \cite[Theorem 2.2.5]{Spanier}), we conclude that $\pi_n(\widetilde X_\mathbf{G})\cong \pi_n(X)$ for  $n\ge 2$, and that $\pi_0(F)$ may be identified with $F$.
Furthermore, we have the following exact sequence of groups and pointed sets
$$\xymatrix{
1 \ar[r] &  \pi_1(\widetilde X_\mathbf{G}) \ar[r]^{{p_\mathbf{G}}_\sharp} &  \pi_1(X) \ar[r]^{\partial} &
\pi_0(F) \ar[r] & \pi_0(\widetilde X_\mathbf{G}) \ar[r] &\pi_0(X) \ar[r] & \ast}$$
where the function $\partial$ is determined by the action of $\pi_1(X)$ on the fibre $F$.
We will normally assume that $X$ is path-connected in which case the above exact sequence ends at the term $\pi_0(\widetilde X_\mathbf{G})$.
The following theorem identifies $\pi_1(\widetilde X_\mathbf{G})$ and $\pi_0(\widetilde X_\mathbf{G})$. 
Observe that the fibre $F$ is a closed subspace of $\widetilde X_\mathbf{G}$, which by Proposition \ref{prop fibre limit} 
may be identified with the inverse limit of the system of cosets $F=\varprojlim \pi_1(X_i)/G_i$. 

\begin{theorem}
\label{thm ex seq}
Assume $X$ is path-connected. Then there is an exact sequence of groups and pointed sets
$$\xymatrix{
1 \ar[r] &  \pi_1(\widetilde X_\mathbf{G}) \ar[r]^{{p_\mathbf{G}}_\sharp} &  \pi_1(X) \ar[r]^-{\partial} &
\varprojlim \pi_1(X_i)/G_i \ar[r] & \pi_0(\widetilde X_\mathbf{G}) \ar[r] & \ast}$$
where the connecting function $\partial$ is given by the composition 
$$\xymatrix{
\pi_1(X) \ar[rr]^-{\varprojlim (u_i)_\sharp} & & \varprojlim \pi_1(X_i) \ar[rr] & & \varprojlim \pi_1(X_i)/G_i}
$$
Consequently, $p_\mathbf{G}$ induces an isomorphism
$$\pi_1(\widetilde X_\mathbf{G}) \cong \bigcap_{i\in\mathcal{I}} (u_i)_\sharp^{-1}(G_i).$$
Moreover, if  $G_i$ is a normal subgroup of $\pi_1(X_i)$ for each $i$, then $\varprojlim \pi_1(X_i)/G_i$ is a group, $\partial$ is a homomorphism and
$\pi_0(\widetilde X_\mathbf{G})$ may be identified with the set of cosets $$\bigl(\varprojlim \pi_1(X_i)/G_i\bigr)/\partial(\pi_1(X)).$$
\end{theorem}
\begin{proof}
By its definition, $\partial$ maps every $\alpha\in\pi_1(X)$ to the end point of the lifting to $\widetilde X_\mathbf{G}$  of any representative of $\alpha$,
$\partial(\alpha)=\widetilde\alpha(1)\in F$. The uniqueness of liftings in covering spaces imply the commutativity of the following diagram 
$$\xymatrix{
\pi_1(X) \ar[r]^\partial \ar[d]_{(u_i)_\sharp} & F \ar[d]^{u_i}\\
\pi_1(X_i) \ar[r]^\partial \ar@{->>}[d] & F_i \ar@{=}[d]\\
\pi_1(X_i)/G_i \ar[r]_-{l_i} & F_i }
$$
which, together with Proposition \ref{prop fibre limit}, leads to the above description of the connecting map $\partial$. 

The description of $\partial$ implies that $(u_i)_\sharp(p_\mathbf{G}(\pi_1(\widetilde X_\mathbf{G})))\subseteq G_i$ for all $i\in\mathcal{I}$, therefore 
$\im p_\mathbf{G}\subseteq \bigcap_i (u_i)_\sharp^{-1}(G_i).$ Conversely, let $\alpha$ be a loop representing an element of $\bigcap_i (u_i)_\sharp^{-1}(G_i)$.
Then for every $i$ the loop $u_i\circ\alpha$ represents an element of $G_i\le\pi_1(X_i)$, so it lifts to a loop $\widetilde \alpha_i$ in $\widetilde X_i/G_i$. 
The uniqueness of liftings imply that $\bar u_{ij}\circ \widetilde \alpha_j=\widetilde \alpha_i$ whenever $i\le j$, hence we obtain a loop 
$\widetilde \alpha:=\varprojlim \widetilde \alpha_i$ in $\widetilde X_\mathbf{G}$. Clearly, $p_\mathbf{G}\circ\widetilde \alpha=\alpha$, therefore $\alpha\in\im p_\mathbf{G}$.
\\[2mm]
The proof of the last claim is straightforward.
\end{proof} 

\ \\
There are two special cases of the Theorem worth mentioning:\\[2mm]
 1.\,For a subgroup $G\le\pi_1(X)$ one may define a coherent thread $\mathbf{G}$ by letting $G_i:=(u_i)_\sharp(G)$. Then $(u_i)_\sharp^{-1}(G_i)=G\cdot \Ker (u_i)_\sharp$,
and so 
$$\pi_1(\widetilde X_G)\cong G\cdot\bigcap_\mathcal{I} \Ker(u_i)_\sharp.$$ 
Moreover if $G$ is a normal subgroup of $\pi_1(X)$, and if all homomorphisms 
$(u_i)_\sharp$ are surjective, then $\varprojlim \pi_1(X_i)/G_i$ is a group and $\pi_0(\widetilde X)$ is a set of cosets.\\[2mm]
 2.\,If $\mathbf{G}$ is a thread of trivial groups then clearly $\widetilde X_\mathbf{G}=\widetilde X$. In this case 
$$\pi_1(\widetilde X)\cong \bigcap_\mathcal{I} \Ker(u_i)_\sharp,$$
while $\pi_0(\widetilde X)$ may be identified with the set of cosets $\varprojlim \pi_1(X_i)/\partial(\pi_1(X))$.

\subsection{Deck transformations}

In the theory of covering spaces deck transformations provide a crucial connection between the algebra of the fundamental group and the geometry of 
the covering space. Their role is even more important in the theory of lifting spaces because they enclose the information about the interleaving
of the path-components of the total space and induce a topology on the fundamental group of the base.

A \emph{deck transformation} of a lifting space $p\colon L\to X$ is a homeomorphism $f\colon L\to L$, such that $p\circ f=p$:
$$\xymatrix{
L\ar[rr]^f \ar[dr]_p & & L \ar[dl]^p \\
& X}
$$
The set of all deck transformations of $p$ clearly form a group, which we denote by $A(p)$. We will consider two basic questions:
is the action of $A(p)$ free and is it transitive on the fibres of $p$. Our first result is valid for general lifting spaces.

\begin{proposition}
\label{prop free action}
Let $p\colon L\to X$ be a lifting space such that each path-component of $L$ is dense in $L$. Then  the action of $A(p)$ on $L$ is free.
\end{proposition}
\begin{proof}
Assume that $f(x)=x$ for some $f\in A(p)$ and $x\in L$. For every path $\alpha\colon (I,0)\to (L,x)$ the equality $p\circ f\circ \alpha=p\circ\alpha$
implies that $f\circ\alpha$ and $\alpha$ are lifts of the path $p\circ\alpha\colon I\to X$. Since $f(\alpha(0))=f(x)=x=\alpha(0)$ and since  
$p$ has unique path liftings, we conclude that $f(\alpha(1))=\alpha(1)$ and so $f$ fixes all elements of the path-component of $L$ that contains $x$. 
But each path-component of $L$ is dense in $L$ and so the deck transformation $f$ must be the identity.
\end{proof}

We must therefore determine when are the path-components of $L$ dense. For inverse limits of coverings it is sufficient to consider the
function $\partial$ that we introduced earlier. 

\begin{proposition}
\label{prop dense leaf}
Let $\mathbf{X}$ be an inverse system of spaces, $\mathbf{G}$ a coherent thread of groups and and let $F$ be the fibre of $p_\mathbf{G}$. 

If the image of the function  $\partial\colon \pi_1(X)\to \pi_0(F)=\varprojlim \pi_1(X_i)/G_i$ is dense in 
$\pi_0(F)$ (with respect to the inverse limit topology) then every path-component of $\widetilde X_\mathbf{G}$ is dense in $\widetilde X_\mathbf{G}$.

Conversely, if $X$ is locally path-connected and every path-component of $\widetilde X_\mathbf{G}$ is dense, then the image of $\partial$ is dense in $F$.
\end{proposition}
\begin{proof}
Recall that we have identified $F$ with $\pi_0(F)$ and that the function $\partial$ is determined by the liftings of representatives of $\pi_1(X)$ in 
$\widetilde X_\mathbf{G}$. This in particular means that all elements of the image of $\partial$ belong to the same path-component $L_0$ of $\widetilde X_\mathbf{G}$.

For a given $\tilde x\in \widetilde X_\mathbf{G}$ let $\alpha$ be a path in $X$ from $x_0$ to $x:=p_\mathbf{G}(\tilde x)$. Then the formula 
$g(\tilde y):=\langle\alpha,\tilde y\rangle(1)$  determines a homeomorphism between $F$ and the fibre $p_\mathbf{G}^{-1}(x)$. This implies that
$\tilde x$ is in the closure of $g(\im \partial)$, and hence in the closure of the path-component $L_0$, because clearly $g(\im \partial)\subset L_0$. 

To show that some other path-component $L_1$ of $\widetilde X_\mathbf{G}$ is also dense, it is sufficient to choose a base-point for $\widetilde X_\mathbf{G}$ 
in $L_1$ and repeat the above argument. Observe that a different choice of a base-point simply conjugates the function  $\partial$, so that its image is still
dense in $F$.

To show the converse statement, assume that some $\tilde x\in F=p_\mathbf{G}^{-1}(x_0)$ is in the closure of the path component $L_0$ and let 
$\widetilde U:=\prod_{i\in\mathcal{F}} \widetilde U_i\times \prod_{i\notin\mathcal{F}}  \widetilde X_i/G_i$ be a product neighbourhood of $\tilde x$.
Here $\mathcal F$ is some finite subset of $\mathcal I$ and each $\widetilde U_i$ is homeomorphically projected by $p_i$ to some elementary neighbourhood 
in $X_i$. We must show that $\widetilde U$ intersects $\im\partial=L_0\cap F$. 
By Proposition \ref{prop open map} the projection $p_\mathbf{G}$ is an open map, so we may find a path-connected open set $V$ such that 
$x_0\in V \subseteq p_\mathbf{G}(\widetilde U)$. Since $\tilde x$ is in the closure of $L_0$ there exists 
$\tilde y\in\widetilde U\cap p_\mathbf{G}^{-1}(V)\cap L_0$. Let $\alpha$ be a path in $V$ from $p_\mathbf{G}(\tilde y)$ to $x_0$. Then $\langle\alpha,\tilde y\rangle(1)$ is
by construction an element of $\im \partial\cap \widetilde U$, which proves that $\im\partial$ is dense in $F$
\end{proof}

For arbitrary lifting spaces it can be sometimes hard to determine whether the image of $\partial$ is dense in the fibre, but for inverse limits of coverings we
may rely on  the following algebraic criterion. Observe that if in an inverse system of sets $(A_i, u_{ij},\mathcal{I})$ we replace each 
$A_i$ by the corresponding  \emph{stable image} $A_i^\St:=\bigcap_{j\ge i} u_{ij}(A_j)$
(and the bonding maps by their restrictions) then $\varprojlim(A_i, u_{ij},\mathcal{I})=\varprojlim(A_i^\St, u_{ij},\mathcal{I})$, i.e. the inverse limit depends only on stable images of bonding maps. In fact, the converse is also true, as the stable image is precisely the image of the projection from the inverse limit, 
$A_i^\St=\varphi_i(\varprojlim A_i)$. 
It is clear that for every morphism from a set $A$ into the system, $(\varphi_i\colon A\to A_i,i\in\mathcal{I})$, we have $\varphi_i(A)\subseteq A_i^\St$ 
for all $i\in\mathcal{I}$. We will say that such a morphism is \emph{stably surjective} if $\varphi_i(A)=A_i^\St$ for all $i\in\mathcal{I}$.
Obviously, every morphism that consists of surjective maps is stably surjective.

\begin{lemma}
\label{lem dense image}
Let $(\varphi_i\colon A\to A_i, \mathcal{I})$ be a stably surjective morphism from $A$ to an inverse system of sets 
$(A_i,u_{ij},\mathcal{I})$. Then  the image of the limit map $\varphi:=\varprojlim \varphi_i\colon A\to \varprojlim A_i$ 
is dense in $\varprojlim A_i$ (with respect to the product topology of discrete spaces). 
\end{lemma}
\begin{proof}
The statement is trivial when the index set $\mathcal{I}$ is finite, so we will assume that $\mathcal{I}$ is infinite.
Let $(a_i)$ be an element of $\varprojlim A_i$. By definition of a product topology, a local basis of neighbourhoods at $(a_i)$ is given by the sets of the form 
$U_\mathcal{F}:=\prod_{i\in \mathcal{F}} \{a_i\}\times \prod_{i\notin \mathcal{F}} {A_i}$, where $\mathcal{F}$ is a finite subset of $\mathcal{I}$. 
Given any such $\mathcal{F}$, let $j\in \mathcal{I}$ be bigger then all elements of $\mathcal{F}$. By the assumptions, there exists $a\in A$ such that 
$\varphi_j(a)=a_j$, but then $\varphi(a)_i=a_i$ for every $i\le j$, and so $\varphi(a)\in U_\mathcal{F}$. We conclude that $\varphi(A)$ is dense in $\varprojlim A_i$.
\end{proof}

We may finally formulate our main result about the free action of the group of deck transformations of an inverse limit of coverings. 

\begin{theorem}
\label{thm st surj}
Let $\mathbf{X}$ be an inverse system of spaces and let $\mathbf{G}$ be a coherent thread of groups. If the morphism 
$((u_i)_\sharp\colon \pi_1(X)\to \pi_1(X_i)/G_i,\mathcal{i})$ is stably surjective, then $A(p_\mathbf{G})$ acts freely on $\widetilde X_\mathbf{G}$. 
\end{theorem}
\begin{proof}
By Theorem \ref{thm ex seq} the function $\partial$ can be identified with the inverse limit $\varprojlim (u_i)_\sharp$. By Lemma \ref{lem dense image} the image
of $\partial$ is dense in the $\varprojlim \pi_1(X_i)/G_i$, which is by Proposition \ref{prop fibre limit} homeomorphic to the fibre  of $p_\mathbf{G}$.
Proposition \ref{prop dense leaf} implies that the path-components of $\widetilde X_\mathbf{G}$ are dense, and finally Proposition \ref{prop free action}
implies that the action of $A(p_\mathbf{G})$ on $\widetilde X_\mathbf{G}$ is free.
\end{proof}

Our next objective is to examine the transitivity of the action of $A(p)$ on the fibre of $p$. It is well-known that for a covering space 
$p\colon \widetilde X/G\to X$ the action of $A(p)$ on the fibre is transitive if and only if $G$ is a normal subgroup of $\pi_1(X)$. 
It is therefore reasonable to restrict our attention to coherent threads of normal subgroups. Let 
$(X_i,u_{ij},\mathcal{I})$ be an inverse system of spaces and $\mathbf{G}=(G_i,\mathcal{I})$ a coherent thread of normal subgroups, i.e. $G_i\triangleleft \pi_1(X_i)$ 
for every $i\in \mathcal{I}$. Then we have group isomorphisms $l_i\colon A(p_i)\to\pi_1(X_i)/G_i$, explicitly given as $l_i(f):=[p\circ\tilde\alpha]$, where 
$\tilde \alpha$ is any path in $\widetilde X_i/G_i$ from the base-point $\tilde x_i$ to its image $f(\tilde x_i)$, and $[p\circ\tilde\alpha]$ is the coset
in $\pi_1(X_i)/G_i$ determined by the loop $p\circ\tilde\alpha$ (cf. the description in \cite[Section 2.6]{Spanier}). Moreover, whenever $j\ge i$ there is a homomorphism
$\hat u_{ij}\colon A(p_j)\to A(p_i)$, where $\hat u_{ij}(f)$ is defined to be the unique deck transformation of $p_i\colon \widetilde X_i/G_i\to X_i$ that maps the base-point
$\tilde x_i$ to $\hat u_{ij}(f(\tilde x_j))$. It is easy to check that whenever $i\le j$ we have a commutative diagram
$$\xymatrix{
A(p_j) \ar[r]^{\hat u_{ij}} \ar[d]_{l_j}^\cong & A(p_i) \ar[d]_{l_i}^\cong  \\
\pi_1(X_j)/G_j \ar[r]_{(u_{ij})_\sharp} & \pi_1(X_i)/G_i}
$$
Thus the homomorphisms $l_i$ determine an isomorphism of inverse systems 
$$(l_i)\colon (A(p_i),\hat u_{ij},\mathcal{I})\to (\pi_1(X_i)/G_i,(u_{ij})_\sharp, \mathcal{I}).$$ 

\begin{theorem}
\label{thm A(p)}
There is an isomorphism of groups 
$$\varprojlim l_i\colon \varprojlim A(p_i)\stackrel{\cong}{\longrightarrow} \varprojlim\pi_1(X_i)/G_i.$$ 
The group  $\varprojlim A(p_i)$ acts freely and transitively on the fibre of $p_\mathbf{G}$. Since 
$\varprojlim A(p_i)$ is a subgroup of $ A(p_\mathbf{G})$ it follows that  $A(p_\mathbf{G})$ also acts transitively on the fibre of $p_\mathbf{G}$.
\end{theorem}
\begin{proof}
That $\varprojlim l_i$ is an isomorphism follows from the above discussion. Toward the proof of transitivity, let $(\tilde x_i),(\tilde x'_i)$ 
be elements of the fibre of $p_\mathbf{G}$: 
then for every $i\in\mathcal{I}$ there exists a unique deck transformation $f_i\in A(p_i)$ such that $f_i(\tilde x_i)=\tilde x'_i$. In order to prove that
transformations $f_i$ represent an element of the inverse limit, consider an element $\tilde y\in \widetilde X_j/G_j$ and a path 
$\widetilde \alpha \colon (I,0,1)\to (\widetilde X_j/G_j, \tilde x_j,\tilde y)$. Then
$$p_i f_i \bar u_{ij}\widetilde\alpha=p_i \bar u_{ij}\widetilde\alpha=u_{ij} p_j\widetilde\alpha=u_{ij} p_j f_j\widetilde\alpha=p_i \bar u_{ij} f_j\widetilde\alpha,$$
which means that $f_i \bar u_{ij}\widetilde\alpha$ and $\bar u_{ij} f_j\widetilde\alpha$ are paths with the same initial point and the same projection in 
$\widetilde X_i/G_i$, so by the monodromy theorem they coincide. In particular $f_i(\bar u_{ij}(\tilde y))=\bar u_{ij}(f_j(\tilde y))$ for every
$\tilde y\in \widetilde X_j/G_j$, therefore $(f_i)$ is an element of $\varprojlim A(p_i)$ that maps $(\tilde x_i)$ to $(\tilde x'_i)$. 

On the other side, if $\varprojlim f_i(\tilde x_i)=(\tilde x_i)$ for some $(f_i)\in\varprojlim A_(p_i)$ and $(x_i)\in\widetilde X_\mathbf{G}$ then $f_i(\widetilde x_i)=x_i$ for 
each $i\in\mathcal{I}$. It follows that all $f_i$ are identity deck transformations on their respective domains, therefore the action of $\varprojlim A(p_i)$
is free. 
\end{proof}

One could expect that $A(p)$ coincides with $\varprojlim A_(p_i)$ but we don't know if that is true in general. In fact, there are inverse limits of coverings 
where the path-components are not dense (like the Warsawonoid from Example \ref{ex Warsawonoid}), and so  it is conceivable that the action
of $A(p)$ may not be free. In view of the above Theorem, this would imply that $\varprojlim A(p_i)$ is a proper subgroup of $A(p)$. This problem disappears, if 
the fundamental group of the limit maps surjectively to the fundamental groups of its approximations, so we have the following result.

\begin{corollary}
\label{cor A(p) on fibre}
Let $\mathbf{X}$ be an inverse system of spaces and let $\mathbf{G}$ be a coherent thread of normal subgroups. If the morphism 
$((u_i)_\sharp\colon \pi_1(X)\to \pi_1(X_i)/G_i,\mathcal{I})$ is stably surjective, then  $A(p_\mathbf{G})$ 
acts freely and transitively on the fibre of $p_\mathbf{G}$ and is therefore isomorphic to $\varprojlim A(p_i)$. 

Furthermore, the map $p=\varprojlim p_i\colon \widetilde X_\mathbf{G}\to X$ induces a continuous bijection 
$\overline p\colon \widetilde X_\mathbf{G}/A(p_\mathbf{G})\to X$. In addition, if $X$ is locally path-connected then $\overline p$ is a homeomorphism.
\end{corollary}
\begin{proof}
The action of $A(p_\mathbf{G})$ on the fibre $F$ is free and transitive by the assumptions and 
Theorems \ref{thm st surj} and \ref{thm A(p)}. Thus we may conclude that there is a bijection $A(p_\mathbf{G})\to F\cong\varprojlim A(p_i)$, which
is clearly compatible with the group structures. 

Map $\overline p$ is induced by $p\colon \widetilde X_\mathbf{G}\to X$ as in the following diagram
$$\xymatrix{
\widetilde X_\mathbf{G} \ar[r]^p \ar[d] & X\\
\widetilde X_\mathbf{G}/A(p_\mathbf{G}) \ar@{-->}[ru]_{\overline{p}} 
}$$
and is clearly a continuous bijection. If $X$ is locally connected then Proposition \ref{prop open map} implies that the map $p$ is open. 
Then by the definition of quotient topology $\overline p$ is also an open map and hence a homeomorphism.
\end{proof}

\ \\

\section{Universal lifting spaces}

We are now going to apply methods developed in previous sections to construct certain lifting spaces that in many aspect behave as the universal 
covering spaces. We will use freely terminology and constructions that are standard in shape theory and are described, for example, in \cite{Mardesic-Segal}.
Recall that all spaces are based (even if the base-points are normally omitted from the notation) and the maps are base-point preserving.

Throughout this section $X$ will be a path-connected metric compactum. We can choose an embedding $X\hookrightarrow M$ into an absolute retract (for metric spaces) $M$, 
and consider the inverse system $\mathbf{X}$ of open neighbourhoods of $X$ in $M$, ordered by the inclusion. Then $\mathbf{X}$ is an \emph{ANR expansion}
of $X$ in the sense of \cite{Mardesic-Segal}, and moreover $X=\varprojlim\mathbf{X}$. Observe that each space in the expansion is semi-locally simply-connected
and therefore admit all covering spaces, including the universal one. As in the previous section, we may define $\widetilde{\mathbf X}$ to be the 
associated inverse system of universal coverings, and let $\widetilde X:= \varprojlim\mathbf{\widetilde X}$. We are
going to prove that the definition of $\widetilde X$ is independent of the embedding $X\hookrightarrow M$. 

To this end we will first show that the above construction is functorial. Let $Y\hookrightarrow N$ be another metric compactum embedded into an absolute retract $N$, 
and let $f\colon X\to Y$ be any map. By the standard properties of the ANR expansions, $f$ induces a morphism of systems $\mathbf{f}\colon \mathbf{X}\to \mathbf{Y}$. 
Note  that the morphism $\mathbf{f}$ is in general not level-preserving, and commutes with bonding maps only up to homotopy. Nevertheless,
when composed with projections from $X$, it commutes exactly, so we may write $f=\varprojlim\mathbf{f}$. Every map between base-spaces admits a unique 
lifting to a map between respective universal coverings, so we have a morphism of systems $\mathbf{\tilde{f}}\colon \mathbf{\widetilde{X}}\to\mathbf{\widetilde{Y}}$ 
that is the unique lifting of $\mathbf{f}$. Here again, the morphism $\mathbf{\tilde{f}}$ commutes with the bonding maps only up to homotopy, but it commutes exactly
when composed with projections from $\widetilde X$, so we are justified to define $\tilde f:=\varprojlim\mathbf{\tilde{f}}$.
Clearly,  $\widetilde 1_X=1_{\widetilde X}$ and $\widetilde{g\circ f}=\widetilde g\circ\widetilde f$. 

\begin{proposition}
\label{functoriality}
The definition of $\widetilde X$ is independent of the choice of embedding for $X$.
The correspondence $X\mapsto \widetilde X$ and $f\mapsto \widetilde f$ determines a functor from
compact metric spaces to metric spaces. This functor is augmented in the sense that the following diagram
is commutative 
$$\xymatrix{
{\widetilde X} \ar[r]^{\tilde f} \ar[d]_{p_X} & {\widetilde Y}  \ar[d]^{p_Y}\\
X \ar[r]^{f} & Y
}$$
\end{proposition} 
\begin{proof}
Let $i\colon X\hookrightarrow M$ and $j\colon X\hookrightarrow N$ be embeddings of $X$ into absolute retracts $M$ and $N$. 
Then we have corresponding ANR expansions $\mathbf{X}_M$ and $\mathbf{X}_N$. Let $\mathbf{f}\colon \mathbf{X}_M\to\mathbf{X}_N$  and 
$\mathbf{g}\colon \mathbf{X}_N\to\mathbf{X}_M$ morphisms of systems induced by the identity map on $X$. By the above discussion
$\mathbf{f}$ and $\mathbf{g}$ induce maps $\tilde f,\tilde g\colon \widetilde X\to\widetilde X $ which are inverse one to the other, which implies that 
$\widetilde X$ is uniquely defined up to a natural homeomorphism. The functoriality and the relation between $f$ and $\tilde f$ follow directly from 
the definitions.
\end{proof}

It is a standard fact of shape theory (see for example \cite[I, 5.1]{Mardesic-Segal}) that for metric compacta every inverse system of polyhedra whose limit is $X$ is 
an expansion and that for every such expansion we can choose a cofinal sequence. Thus we obtain a more manageable description of $\widetilde X$: 

\begin{corollary}
If $X$ is a metric compactum then 
$\widetilde X=\varprojlim \widetilde X_i$ for any inverse sequence of polyhedra, converging to $X$.
\end{corollary}

Clearly, if $X$ is a compact polyhedron, then $\widetilde X$ is just the universal covering space of $X$. For a less trivial example, if $W$ 
denotes the Warsaw circle, then it is easy to see that $\widetilde W$ is the \emph{'Warsaw line'}, which we may describe as the real line in which every
segment $[n,n+1)$ is replaced by the topologist's sine curve. As another illustration, if $H$ is the Hawaiian earring, then $\widetilde H$ is the inverse limit of trees
that are universal covering spaces for finite wedges of circles.

In order to explain in what sense is $\widetilde X$ universal with respect to the  lifting spaces that are inverse limits of coverings, 
let us consider an expansion $\mathbf{X}=(X_i,u_{ij},\mathcal{I})$ of a metric compactum $X$. 
Then any inverse system of coverings over $\mathbf{X}$ uniquely corresponds to a choice of a coherent thread $\mathbf{G}=(G_i,\mathcal{I})$, i.e.
is of the form $(\widetilde X_i/G_i,\bar u_{ij},\mathcal{I})$. 
By the lifting properties of covering spaces, for every $i\le j$ in $\mathcal{I}$ there are unique maps $q_i,q_j$ for which the following diagram 
commutes
$$\xymatrix{
          & \widetilde X_j \ar[rr]^{\tilde u_{ij}} \ar[dl]_{q_j} \ar[ddl]^(.3){\tilde q_j} & & \widetilde X_i \ar[dl]_-{q_i} \ar[ddl]^(.3){\tilde q_i}\\
 \widetilde X_j/G_j \ar[rr]_{\bar u_{ij}} \ar[d]_{p_j} & &  \widetilde X_i/G_i  \ar[d]_{p_i} \\
  X_j \ar[rr]_{u_{ij}} & & X_i
}$$
Thus we obtain a commutative diagram of inverse systems and the corresponding limits
$$\xymatrix{
\mathbf{\widetilde X}_{\mathbf G} \ar[d]_{\mathbf p} & \mathbf{\widetilde X} \ar[l]_{\mathbf q} \ar[ld]^-{\mathbf{\tilde q}} & & 
{\widetilde X}_{\mathbf G} \ar[d]_{p} & {\widetilde X} \ar[l]_{q} \ar[ld]^-{\tilde q}\\
\mathbf X & & & X}
$$

\begin{theorem}
\label{thm dense image}
Let $\tilde q\colon \widetilde X\to X$ be the universal lifting space for $X$ and let $p\colon \overline X\to X$ be a lifting space 
obtained as a limit of an inverse system of coverings over some expansion of $X$.
Then $\overline X=\widetilde X_\mathbf{G}$ for some (essentially unique) coherent thread $\mathbf{G}$, and there is a unique (base-point preserving) 
lifting space projection $q\colon \widetilde X\to \widetilde X_\mathbf{G}$ for which $p\circ q=\widetilde q$. 

Furthermore, let $\pi_1(X_i)^\St$ and $(\pi_1(X_i)/G_i)^\St$ be the stable images of the systems of groups and cosets
induced by the inverse systems for $X$ and $\overline X$. Then $q(\widetilde X)$ is dense in $\widetilde X_\mathbf{G}$ if and only if  
the induced map 
$$\pi_1(X_i)^\St\longrightarrow(\pi_1(X_i)/G_i)^\St$$ 
is surjective for every $i\in \mathcal{I}$.
\end{theorem}
\begin{proof}
The first assertion follows from the above discussion, so it remains to prove the characterization of density. Assume that the map between stable images
is surjective and consider an element $\bar x=(\tilde x_iG_i)\in\overline X$ (with $\tilde x_i\in\widetilde X_i$ as representatives of the orbits).
By choosing a path $\alpha$ in $X$ connecting the base-point to $p(\bar x)$ we obtain a coherent sequence of paths $(\alpha_i\colon I\to X_i)$, where $\alpha_i$ 
connects the base point of $X_i$ to $u_i(p(\bar x))$. The unique path-lifting along $\alpha_i$ yields a commutative diagram 
$$\xymatrix{
\pi_1(X_i) \ar@{^(->}[r] \ar[d] &  \widetilde X_i \ar[d]^{q_i} \\
\pi_1(X_i)/G_i \ar@{^(->}[r] \ar[d] &  \widetilde X_i/G_i \ar[d]^{p_i} \\
{\ \ast\ } \ar@{^(->}[r]  &   X_i  
}$$
where the horizontal maps are $\pi_1(X_i)$-equivariant bijections onto the fibres over $u_i(p(\bar x))$ of the lifting spaces $p_i$ and $\tilde q_i$.
By construction, these diagrams commute with the bonding maps of the respective inverse systems so they form a commutative diagram of inverse systems
$$\xymatrix{
(\pi_1(X_i),(\tilde u_{ij})_\sharp,\mathcal{I}) \ar[r] \ar[d] &  (\widetilde X_i,\tilde u_{ij},\mathcal{I}) \ar[d]^{q_i} \\
(\pi_1(X_i)/G_i,(\bar u_{ij})_\sharp,\mathcal{I}) \ar[r] \ar[d] &  (\widetilde X_i/G_i,\bar u_{ij},\mathcal{I}) \ar[d]^{p_i} \\
{\ \ast\ } \ar[r]  &   (X_i,u_{ij},\mathcal{I})  
}$$
As a consequence, we obtain bijections $\pi_1(X_i)^\St\to \widetilde X_i^\St\cap \tilde q_i^{-1}(u_i(p(\bar x)))$ and 
$(\pi_1(X_i)/G_i)^\St\to (\widetilde X_i/G_i)^\St\cap \tilde p_i^{-1}(u_i(p(\bar x)))$.

A local basis of neighbourhoods at $(\tilde x_iG_i)$ is given by the sets of the form 
$U_\mathcal{F}:=\prod_{i\in \mathcal{F}} \{x_iG_i\}\times \prod_{i\notin \mathcal{F}} {\widetilde X_i/G_i}$, where $\mathcal{F}$ is any finite subset of $\mathcal{I}$. 
Let $j\in\mathcal{I}$ be bigger than all elements of $\mathcal{F}$. Since $x_jG_j\in(\widetilde X_j/G_j)^\St$, it corresponds to 
to some $g_jG_j$, where by our assumption $g_j\in\pi_1(X_j)^\St$. It follows that there is an $\tilde x\in \widetilde X$ such that 
$\tilde u_j(\tilde x)G_j=\tilde x_jG_j$, and consequently $\tilde u_i(\tilde x)G_i=\tilde x_iG_i$ for all $i\in\mathcal{F}$. 
We have thus proved that $q(\widetilde X)$ intersects every open set in $\widetilde X_\mathbf{G}$, hence the image of $q$ is dense in $\widetilde X_\mathbf{G}$.

Conversely, if for some $i\in\mathcal{I}$ the map $\pi_1(X_i)^\St\longrightarrow(\pi_1(X_i)/G_i)^\St$ is not surjective, then we may use the previously 
described correspondence between $\pi_1(X_i)$ and $\pi_1(X_i)/G_i$ with the fibres of $q_i$ and $p_i$ to find an element $\bar x\in\widetilde X_\mathbf{G}$, 
such that $q(\widetilde X)$ does not intersect its open neighbourhood $\{\bar u_i(\bar x)\}\times \prod_{j\ne i} {\widetilde X_j/G_j}$.
\end{proof}

In various practical situations it may be hard to verify whether the image of $\widetilde X$ is dense in $\overline X$ directly from the theorem. 
The following corollary provides several sufficient conditions. 

\begin{corollary}
\label{cor dense image}
Any of the following conditions imply that $q(\widetilde X)$ is dense in $\widetilde X_\mathbf{G}$.
\begin{enumerate}
\item The inverse system $(\pi_1(X_i), (u_{ij})_\sharp,\mathcal{I})$ has the Mittag-Leffler property
 (this holds in particular, if all bonding homomorphisms in the system are surjective) and $\mathbf{G}$ is arbitrary.
\item $G_i$ is a subgroup of $\pi_1(X_i)^\St$ for every $i\in\mathcal{I}$.
\item $G_i\cdot \pi_1(X_i)^\St=\pi_1(X_i)$ for every $i\in\mathcal{I}$.
\item There exists a cofinal subsequence $\mathcal{C}\subseteq\mathcal{I}$, and all groups in the thread $\mathbf{G}$ are finite.  
\end{enumerate}
\end{corollary}
\begin{proof}
\begin{enumerate}
\item 
Recall that $(\pi_1(X_i), (u_{ij})_\sharp,\mathcal{I})$ satisfy the Mittag-Leffler condition if  the stable images are achieved at some
finite stage, i.e. for every $i\in\mathcal{I}$ there is a $j\ge i$ such that $\pi_1(X_i)^\St= (u_{ik})_\sharp(\pi_1(X_k))$ for every $k\ge j$. 
Given an element $g_iG_i\in\pi_1(X_i)^\St$  there exists  $g_j\in \pi_1(X_j)$ such that 
$g_iG_i=(u_{ij})_\sharp(g_jG_j)=(u_{ij})_\sharp(g_j)G_i$, so $g_iG_i$ is the image of $(u_{ij})_\sharp(g_j)$, 
which is by the Mittag-Leffler condition an element of $\pi_1(X_i)^\St$.
\item 
Assume that $g_iG_i\in (\pi_1(X_i)/G_i)^\St$, so that for every $j\ge i$ there exists some $g_j\in\pi_1(X_j)$ satisfying 
$g_iG_i=(u_{ij})_\sharp(g_jG_j)=(u_{ij})_\sharp(g_j)G_i$. It follows that $(u_{ij})_\sharp(g_j)=g_ig$ for some $g\in G_i$. 
Since $G_i\subseteq \pi_1(X_i)^\St$ there exists some $h\in \pi_1(X_j)$ such that $(u_{ij})_\sharp(h)=g$, and so 
$g_i=(u_{ij})_\sharp(g_jh^{-1})$, which shows that $g_i$ is in the image of $(u_{ij})_\sharp$ for every $j\ge i$, therefore it is in 
the stable image.
\item
The condition implies that every element of $\pi_1(X_i)/G_i$, and hence every element of $(\pi_1(X_i)/G_i)^\St$, is in the image of $\pi_1(X_i)^\St$.
\item 
If $g_iG_i\in (\pi_1(X_i)/G_i)^\St$, then for every $j\in\mathcal{C}, j\ge i$ there exists $g_j\in\pi_1(X_j)$ such that $(u_{ij})_\sharp(g_jG_j)=g_iG_i$. 
This means that $(u_{ij})(g_j)$ is one of the finitely many elements of the coset $x_iG_i$ and so at least one of them must appear infinitely many 
times as value of $(u_{ij})_\sharp$ for $j\in\mathcal{C}, j\ge i$. By the cofinality of $\mathcal{C}$ this element must be in $\pi_1(X_i)^\St$.
\end{enumerate}
\end{proof}

\begin{example}
If $X$ is a polyhedron, then we may always  take the trivial expansion so the bonding maps on the associated system of fundamental groups are identity homomorphisms. 
If $p\colon\overline X\to X$ is a limit of an inverse system coverings over $X$ then by Theorem \ref{thm dense image} there exists a unique lifting projection 
$q\colon \widetilde X\to \overline X$ such that $p\circ q=\tilde q\colon\widetilde X\to X$, the standard  universal covering space projection. 
Moreover, by Corollary \ref{cor dense image} the image $q(\widetilde X)$ is dense in $\overline X$.
\end{example}

Observe that in principle the density of $q(\widetilde X)$ in $\overline X$ depends on the properties of the expansion used in the construction of $\overline X$. 
However, we have the following result that allows to avoid this objection. 
Our argument is based on the fact that the limit of the inverse system of fundamental groups associated to an expansion of $X$ actually depends only on the space 
itself: the \emph{shape fundamental group} is defined as $\check{\pi}_1(X):=\varprojlim (\pi_1(X_i),(u_{ij})_\sharp, \mathcal{I})$, 
where $(X_i,u_{ij},\mathcal{I})$ is any expansion of $X$ (cf. \cite[II, 3.3]{Mardesic-Segal}). 

\begin{lemma}
\label{lem M-L expansions}
Let $(X_i,u_{ij},\mathcal{I})$ and $(X'_i,u'_{ij},\mathcal{I'})$ be two polyhedral expansions for $X$, and assume that the bonding homomorphisms in the associated system
of fundamental groups $(\pi_1(X_i),(u_{ij})_\sharp,\mathcal{I})$ are surjective. Then the system $(\pi_1(X'_i),(u'_{ij})_\sharp,\mathcal{I'})$ satisfies the 
Mittag-Leffler condition.
\end{lemma}
\begin{proof}
By the properties of expansions, for a given $i'\in\mathcal{I'}$ we may find an $i\in\mathcal{I}$ and a map $v\colon X_i\to X_{i'}$ such that $v\circ u_i=u'_{i'}$.
Similarly, we may find some $j'\in\mathcal{I'}$ and a map $w\colon X_{j'}\to X_i$ satisfying $w\circ u'_{j'}=u_i$. 
$$\xymatrix{
X \ar@{=}[d] \ar[rr]^{u_i} & & X_i \ar[rd]^v \\
X \ar[r]_{u'_{j'}} & X_{j'} \ar[rr]_{u'_{i'j'}} \ar[ru]^w & & X_{i'}
}$$
Then, by applying fundamental groups we obtain the following diagram
$$\xymatrix{
\check{\pi}_1(X)=\varprojlim \pi_1(X_i) \ar@{=}[d] \ar@{->>}[rr]^{(u_i)_\sharp} & & \pi_1(X_i) \ar[rd]^{v_\sharp} \\
\check{\pi}_1(X)=\varprojlim \pi_1(X'_i) \ar[r]_-{(u'_{j'})_\sharp} & \pi_1(X_{j'}) \ar[rr]_{(u'_{i'j'})_\sharp} \ar[ru]^{w_\sharp} & & \pi_1(X_{i'})
}$$
Since $(u_i)_\sharp$ is surjective, so must be $w_\sharp$, thus we have
$$\pi_1(X'_{i'})^\St=(u_{i'})_\sharp(\check{\pi}_1(X))=v_\sharp(\pi_1(X_i))=(u'_{i'j'})_\sharp(\pi_1(X_{j'}))$$ 
which proves that the stable image coincides with the image of the bonding map $(u'_{i'j'})_\sharp$, as required by the Mittag-Leffler condition.
\end{proof}

\begin{example}
We have already mentioned that the universal lifting space of the Warsaw circle $W$ is the Warsaw line $\widetilde W$. Since we may obtain $W$ as a limit
of shrinking annuli, where the induced homomorphisms on the fundamental group are identities, 
the above results imply that $\widetilde W$ is mapped densely into every inverse limit of coverings over any expansion of $W$. This applies in particular to all
Warsawonoids that we described in Example \ref{ex Warsawonoid}. 
\end{example}

In general one cannot expect to find an expansion of $X$ for which the bonding homomorphisms in the induced system of fundamental groups are surjective. 
Nevertheless, this important property can be always achieved for expansions of locally path-connected spaces.

\begin{lemma}
\label{lem LPC expansion}
Every locally path-connected space $X$ admits an expansion $(X_i,u_{ij},\mathcal{I})$ such that the induced homomorphisms $(u_i)_\sharp\colon\pi_1(X)\to \pi_1(X_i)$ 
are surjective for all $i\in\mathcal{I}$.
\end{lemma}
\begin{proof}
We are going to exploit the technique used in the proof of Lemma \ref{lem ex seq}. In fact, let $\mathcal{U}$ be a finite, non-redundant open cover of $X$ 
and let $f\colon X\to |\mathcal{U}|$ the map induced by some choice of a partition of unity subordinated to $\mathcal{U}$. Then 
by Lemma \ref{lem ex seq} the induced homomorphism $f_\sharp\colon\pi_1(X)\to \pi_1(|\mathcal{U}|$ is surjective. Therefore, if we take the standard \v Cech 
expansion of $X$ by polyhedra $X_i$ that are nerves of covers of $X$ by metric balls of radius $1/i$ for $i=1,2,3,\ldots$, 
then the resulting inverse sequence satisfy the desired surjectivity property. 
\end{proof}

As a consequence we obtain strong surjectivity properties of inverse systems of fundamental groups associated to expansions of locally path-connected spaces.

\begin{corollary}
\label{cor stab surjective}
Let $(X_i,u_{ij},\mathcal{I})$ be any expansion of a connected and locally path-connected compact metric space $X$. Then the induced morphism  
$(\pi_1(X)\to\pi_1(X_i),i \in \mathcal{I})$ is stably surjective and the associated inverse system 
$(\pi_1(X_i),(u_{ij})_\sharp,\mathcal{I})$ satisfies the Mittag-Leffler condition.
\end{corollary}
\begin{proof}
By Lemma \ref{lem LPC expansion} and the properties of an expansion we may find for each $i\in\mathcal{I}$ a polyhedron $P$ and maps $v\colon X\to P$ and 
$v_i\colon P\to X_i$ so that $v_i\circ v=u_i$ and $v_\sharp\colon \pi_1(X)\to\pi_1(P)$ is a surjection. On the other side, we may also find a $j\ge i$ and 
a map $v_j\colon X_j\to P$ so that $v_j\circ u_j=v$ and $v\circ v_j\simeq u_{ij}$. Thus we obtain the following diagram
$$\xymatrix{
X\ar[r]^{u_j} \ar[rd]_v & X_j \ar[r]^{u_{ij}} \ar[d]^-{v_j} & X_i\\
  & P \ar[ru]_{v_i}}
$$
which induces a commutative diagram of fundamental groups
$$\xymatrix{
\pi_1(X) \ar[r]^{(u_j)_\sharp} \ar@{->>}[rd]_{v_\sharp} & \pi_1(X_j) \ar[r]^{(u_{ij})_\sharp} \ar[d]^-{(v_j)_\sharp} & \pi_1(X_i)\\
  & \pi_1(P) \ar[ru]_{(v_i)_\sharp}}
$$
Since $(u_{ij})_\sharp$ factors through $(v_i)_\sharp$, it follows that $(v_i)_\sharp(\pi_1(P)$ contains the stable image $\pi_1(X_i)^\St$. 
But then the surjectivity of $v_\sharp$ implies that the homomorphism $(u_i)_\sharp\colon \pi_1(X)\to \pi_1(X_i)^\St$ is also surjective. 
\ \\[2mm]
The second claim follows immediately by Lemma \ref{lem M-L expansions}.
\end{proof}

Connected and locally path-connected compact metric spaced form a large class of spaces that include all finite polyhedra, compact manifolds and many other
important spaces. They are in fact more commonly known as \emph{Peano continua}. This name is a distant echo of the Peano space-filling curves, consolidated by
the famous Hahn-Mazurkiewicz Theorem that characterizes Peano continua 
as Hausdorff spaces that can be obtained as a continuous image of an arc.  The following theorem summarizes the main properties
of lifting spaces over Peano continua.

\begin{theorem}
\label{thm Peano continua}
Let $X$ be a Peano continuum, $\tilde q\colon \widetilde X\to X$ its universal lifting space and $p\colon\overline X\to X$ any lifting space that can be 
obtained as a limit of an inverse system of coverings $p_i\colon \overline X_i\to X_i$ over some expansion of $X$. Then: 
\begin{enumerate}
\item 
There exists a unique lifting projection $q\colon \widetilde X\to \overline X$
such that $p\circ q=\tilde q$. Moreover, the image $q(\widetilde X)$ is dense in $\overline X$.
\item  
The group of deck transformations $A(p)$ acts freely on $\overline X$. 
\item
If $\overline X$ is an inverse limit of normal coverings  then $A(p)$  acts freely and transitively on the fibres of $p$, 
and there is an isomorphism $A(p)\cong \varprojlim A(p_i)$. In particular, $A(\tilde q)$ is naturally isomorphic with the shape 
fundamental group $\check{\pi}_1(X)$.
\item 
If $\overline X$ is an inverse limit of normal coverings then $p$ induces a homeomorphism $\bar p\colon \overline X/A(p)\to X$. 
\item
There is an exact sequence of groups and sets 
$$\xymatrix{
1 \ar[r] & {\pi_1(\widetilde X)} \ar[r]^{(\tilde{q})_\sharp} & \pi_1(X) \ar[r]^\partial & {\check{\pi}_1(X)} \ar[r] &
\pi_0(\widetilde X) \ar[r] & \ast}$$
In particular, $\pi_1(\widetilde X)$ may be identified with the kernel of the natural homomorphism $\partial\colon\pi_1(X)\to\check{\pi}_1(X)$, also known
as \emph{shape kernel} of $X$. Similarly, $\pi_0(\widetilde X)$ may be identified with the \emph{shape cokernel} of $X$, namely the set of cosets
$\check{\pi}_1(X)/\partial(\pi_1(X))$. 
\item 
Let $\mathbf{G}=(G_i)$ be the coherent thread of groups, determined by $\overline X$. Then the map $q\colon \widetilde X\to \overline X$ induces a commutative ladder:
$$\xymatrix{
1 \ar[r] & {\pi_1(\widetilde X)} \ar[d]_{q_\sharp} \ar[r]^{\tilde{q}_\sharp} & \pi_1(X)\ar@{=}[d] \ar[r]^\partial & {\check{\pi}_1(X)} \ar[r] \ar[d] &
\pi_0(\widetilde X) \ar[r] \ar[d]_{q_\sharp} & \ast\\
1 \ar[r] & {\pi_1(\overline X)} \ar[r]_{p_\sharp} & \pi_1(X) \ar[r]_-\partial & {\varprojlim \pi_1(X_i)/G_i} \ar[r] &
\pi_0(\overline X) \ar[r] & \ast
}$$
\end{enumerate}
\end{theorem}
\begin{proof}
The existence of the map $q\colon \widetilde X\to \overline X$ was proved in Theorem \ref{thm dense image}. 
As for the second claim, observe that by Lemma \ref{lem LPC expansion} $X$ admits an expansion such that the induced homomorphisms 
between fundamental groups are surjective. By Corollary \ref{cor stab surjective} the induced system satisfies the Mittag-Leffler condition,
and then by Corollary \ref{cor dense image} it follows that  $q(\widetilde X)$ is dense in $X$.

The statements 2., 3. and 4. also follow from Corollary \ref{cor stab surjective}, combined with Theorem \ref{thm st surj} and Corollary \ref{cor A(p) on fibre}.

Finally, 5. and 6. follow from Theorem \ref{thm ex seq}, in particular from the naturality of the exact sequence of a fibration.
\end{proof}

Observe that the shape kernel and shape cokernel are not shape invariants: in fact they are more like 'anti-invariants' as they measure the variation
of the structure of the universal lifting space within shape-equivalent spaces.

\

We conclude the section with a lifting theorem for inverse limits of covering spaces. Given a map $f\colon X\to Y$ and 
lifting spaces $p\colon {\widetilde X} \to X$ and $q\colon {\widetilde Y}\to Y$ we would like to know if there exists a map
$\tilde f$ for which the following diagram commutes.
$$\xymatrix{
{\widetilde X} \ar@{-->}[r]^{\tilde f} \ar[d]_{p} & {\widetilde Y}  \ar[d]^{q}\\
X \ar[r]^{f} & Y
}$$
If $\widetilde X$ is connected and locally path connected then the answer is given by the classical lifting theorem \cite[Theorem II.4.5]{Spanier}: 
$\widetilde f$ exists if, and only if $f_\sharp(p_\sharp(\pi_1({\widetilde X}))\subseteq q_\sharp(\pi_1(\widetilde Y))$. We are going to 
extend this result to more general lifting spaces. Observe that a very special case was already considered as a part
of Proposition \ref{functoriality}.

Let us first show that the property of being the limit of a sequence of covering spaces over a polyhedral expansion for $X$ 
is independent on the choice of the expansion. 

\begin{proposition}
\label{prop: invlim coverings}
Let $X$ be the limit of an polyhedral expansion $\mathbf{X}=(X_i,u_{ij}\colon X_j\to X_i,i,j\in\NN)$ and let 
Let $p_\mathbf{G}\colon \widetilde X_\mathbf{G}\to X$ be the lifting space determined by a coherent thread $\mathbf{G}$ of 
subgroups of $\pi_1(X_i)$. Then for every polyhedral expansion $\mathbf{P}=(P_i,v_{ij}\colon P_j\to P_i,i,j\in\NN)$ for $X$ 
there exists a coherent thread $\mathbf{H}$, such that $\widetilde X_\mathbf{H}=\widetilde X_\mathbf{G}$ and 
$p_\mathbf{H}=p_\mathbf{G}$.
\end{proposition}
\begin{proof}
We will use properties of polyhedral expansions to obtain approximations of the identity map of $X$ with respect to 
the expansions $\mathbf{X}$ and $\mathbf{P}$ (cf. \cite{Mardesic-Segal}). 
For every $i$ there exists $j$ and a map $f_i\colon P_j\to X_i$ such that 
$f_iv_j=u_i$. Furthermore, there exists a $k$ and a map $g_j\colon X_k\to P_j$ such that $g_ju_k=v_j$. Finally, since 
$f_i$ and $g_j$ are both approximations of the identity on $X$ there exists some index $l$ such that 
$u_{il}\simeq f_ig_ju_{kl}$. 
$$\xymatrix{
X \ar@{=}[dd] \ar[r]^{u_l} &X_l \ar@{=}[dd]\ar[rrr]^{u_{il}} & & & X_i\\
& & & P_j \ar[ur]_{f_i}\\
X\ar[r]_{u_l} & X_l \ar[r]_{u_{kl}} & X_k \ar[ur]_{g_k}
}$$
Let $H_j:=(f_{i\sharp})^{-1}(G_i)\le\pi_1(P_j)$. Then $(g_ku_{kl})_\sharp(G_l)\subseteq H_j$ and we obtain 
a sequence of covering projections
$$\xymatrix{
\widetilde X_l/G_l \ar[r] \ar[d]_{p_l} & \widetilde P_j/H_j \ar[d]_{q_j} \ar[r] & \widetilde X_i/G_i\ar[d]_{p_i}\\
X_l\ar[r]_{g_ku_{kl}} & P_j \ar[r]_{f_i} & X_i
}$$
which shows that we may refine the system of coverings over $X_i$ by coverings over $P_i$. It follows that $\widetilde X_\mathbf{G}$
can be represented as an inverse limit of coverings over the system $\mathbf{P}$ with respect to some coherent thread of groups 
$\mathbf{H}$.
\end{proof}

\begin{theorem}
\label{thm:lifting criterion}
Let $\mathbf{X}$ be a polyhedral expansion for $X$ and $p_\mathbf{G}\colon \widetilde X_\mathbf{G}\to X$ 
the lifting space determined by a coherent thread $\mathbf{G}$. Similarly, let  
$\mathbf{Y}$ be a polyhedral expansion for $Y$ and $q_\mathbf{H}\colon \widetilde Y_\mathbf{H}\to Y$ 
the lifting space determined by a coherent thread $\mathbf{H}$. 

Then a map $f\colon X\to Y$ can be lifted to a map
$\tilde f\colon \widetilde X_\mathbf{G}\to\widetilde Y_\mathbf{H}$ if, and only if for a morphism 
$\mathbf{f}\colon\mathbf{X}\to\mathbf{Y}$ induced by $f$ we have that for every index $i$ there exists some index $j$ such that
$(f_iu_{ij})_\sharp(G_j)\le H_i$.

The condition for the existence of a lifting is independent from the choices of expansions for $X$ and $Y$. 
\end{theorem}
\begin{proof}
The assumption that for every $i$ there is some $j$, such that $(f_iv_{ij})_\sharp(H_j)\le G_i$ implies that there exist maps 
$\tilde f_i$ so that the following diagram commutes 
$$\xymatrix{
\widetilde X_j/G_j \ar[r]^{\tilde f_i} \ar[d]_{p_j} & \widetilde Y_i/H_i\ar[d]^{q_i}\\
X_j \ar[r]_{f_iu_{ij}} & Y_i
}$$
After suitable reindexing we see that maps $\tilde f_i$ determine a morphism of systems 
$\tilde{\mathbf{f}}\colon\widetilde{\mathbf{X}}/\mathbf{G}\to\widetilde{\mathbf{Y}}/\mathbf{H}$, and hence define the lifting 
$\tilde f\colon \widetilde X_\mathbf{G}\to \widetilde Y_\mathbf{H}$ for $f$. 

That the existence of the lifting is independent from the choice of the expansions for $X$ and $Y$ follows from 
Proposition \ref{prop: invlim coverings}. We prefer to omit technical details. The converse implication is obvious.
\end{proof}

\ \\

\section{When is the hat space a lifting space?}

In the preceding sections we have mostly assumed that the spaces under consideration are locally path-connected.
For more general spaces we may first construct the hat space mentioned in Example \ref{ex hat}. 
Since the unique path-lifting property of the map $\iota\colon \hX\to X$ is obvious, $\iota$ is a lifting
space if and only if it is a fibration. When this happens,
every lifting space over $\hX$ is automatically a lifting space over $X$. Moreover, every lifting projection 
$p\colon Y\to X$ with $Y$ locally path-connected, factors through $\hX$ as a composition of two lifting projections.

In this section we discuss in detail the fibration properties of the hat construction. 
In particular we prove that the hat construction over a metric space yields
a fibration for the class of all metric (in fact 1-countable) spaces if and only if the hat space is 
locally compact. 

Recall that the \emph{hat space} $\hX$ is obtained from a space $X$ generated by taking the path components 
of open sets in the topology of $X$ as a sub-basis.  
The identity map $\iota\colon\hX \to X$ is continuous and bijective. If $f\colon Y\to X$ is 
a continuous map and $C$ a path-component of an open set $U\subseteq X$, then 
$f^{-1}(C)$ is a union of components of the open set $f^{-1}(U)$. Therefore, if $Y$ is locally 
path-connected, then $f\colon Y\to \hX$ is also continuous. 
In particular the paths in $X$ correspond precisely to paths in $\hX$ and the same holds for 
path-components of subsets as well. This implies that $\hX$ is 
locally path-connected, so $\widehat{(\hX)}=\hX$. In addition,  if $f\colon Y\to X$ is continuous, 
then $\hat f\colon \widehat Y\to \hX$ (where $\hat f=f$ as 
a function  between sets) is also continuous, and the following diagram commutes
$$\xymatrix{
{\widehat Y} \ar[r]^{\widehat f} \ar[d]_-{\iota} & {\widehat X} \ar[d]^-{\iota}\\
Y \ar[r]_f & X
}$$
We may summarize these facts in categorical terms by saying that the hat construction is an idempotent augmented functor. 

What can be said about the fibration properties of the hat construction? Since maps and homotopies 
from cubes have unique liftings, the projection 
$\iota\colon\hX\to X$ is a  Serre fibration with the unique path-lifting property. In particular, 
this shows that $X$ and $\hX$ have isomorphic homotopy groups. However, the following 
example show that it is not in general a Hurewicz fibration.

\begin{example}
Let us  denote by $\s:=\{1/n\mid n \in \mathbb{N}\} \cup \{0\}\subseteq \mathbb{R}$ the 'model' 
convergent sequence, and let $C\s$ be the cone over $\s$, 
which we view as a subspace of the plane, e.g. $C\s:=\{(t,tu)\in\RR^2\mid t\in[0,1], u\in\s\}$. 
It is easy to see that $\widehat{C\s}$ can be identified
with a countable one-point union of intervals. Consider the homotopy $H\colon C\s\times I\to C\s$, given 
by $H\big((x,y),t\bigl):=(tx,ty)$. Then the constant map $H_0$ lifts to $\widehat{C\s}$ but the 
entire homotopy $H$ cannot be lifted,
because its final stage would give a continuous inverse to $\iota\colon\widehat{C\s}\to C\s$, which would
imply that $C\s$ is locally path-connected, a contradiction.
\end{example}

Is there some natural condition that would imply that $\iota\colon\hX\to X$ is a fibration for 
a sufficiently large class of
spaces, like the metric spaces? 
Toward an answer to this question we prove that in order to check 
the fibration property for metric spaces it is sufficient to check that $\iota$ has the covering
homotopy property for maps from the model sequence $\s$.

\begin{lemma}
\label{hlpslem}
The canonical map $\iota\colon\hX \to X$ 
is a fibration for the class of first countable spaces if and only if it has the homotopy lifting property 
for maps from $\s$ to $X$.
\end{lemma} 
\begin{proof}
One direction is immediate. For the other implication assume that $Y$ is a 1-countable space and that we are
given a homotopy $H\colon Y\times I\to X$ such that the lifting $\widehat H_0\colon Y\times 0\to \hX$
is continuous. In order to prove that $\widehat H\colon Y\times I\to \hX$ is continuous we take a sequence
$\{(y_i,t_i)\}$ in $Y\times I$ converging to $(y,t)$. The sequence $\{y_i\}$ converges to $y$ and hence 
determines a map $g\colon \s\to Y$. The initial stage of the homotopy $F:=H\circ(g\times 1)\colon \s\times I\to X$
lifts to the continuous map $\widehat F_0=\widehat H\circ g\colon Y\times 0\to \hX$, so by the assumption
$\widehat F\colon \s\times I\to \hX$ is continuous as well. In particular the sequence $\widehat F(y_i,t_i)$ 
converges to $\widehat F(y,t)$.
\end{proof}

Recall that a space is \emph{sequentially compact} if every sequence in it has a convergent subsequence. 
In general the compactness and the sequential compactness are not directly related as none of them implies
the other. However, they coincide for the class of metric spaces and for 1-countable spaces
(or more generally, for \emph{sequential spaces}) the compactness implies sequential compactness.   
Moreover, a space is \emph{locally sequentially compact} if every point has a neigborhood whose closure
is sequentially compact.

\begin{theorem}
\label{lclem}
If $X$ is a Hausdorff space and if $\hX$ is locally sequentially compact, then the canonical 
map $\iota\colon\hX \to X$  has the homotopy lifting property for maps from $\s$.
\end{theorem}
\begin{proof}
Let $F\colon \s\times I\to X$ be a map, such that the restriction of the unique lift $\widehat F\colon\s\times I\to\hX$
is continuous when restricted to $\s\times 0$. We are going to show that $\widehat F$ is also continuous. 

First observe that the local path-connectedness of $(\s-0)\times I$ imply that  
$\widehat F$ is continuous on $(\s-0)\times I$, so it only remains to prove the continuity of $\widehat F$ on 
$0\times I$. Since $\widehat F$ is continuous on $\s\times 0$ we may define 
$$t:=\sup\{s\in I\mid \widehat F|_{\s\times [0,s]}\ \text{is continuous}\}.$$
By the local sequential compactness of $\hX$ and the definition of the hat-topology we may choose 
a neighborhood  $U\subset \hX$ of $\widehat F(0,t)$ that is a path-component of an open set $V\subset X$, and 
whose closure is sequentially compact. By the continuity of $F$ there exist $\varepsilon,\delta>0$ 
such that the box neighborhood $B:=(\s\cap [0,\delta))\times (t-\varepsilon,t+\varepsilon)$ is contained in 
$F^{-1}(V)$. Moreover, since the restriction $\widehat F|_{0\times I}$ is continuous, we may assume (by decreasing
$\varepsilon$, if necessary) that $\widehat F(0\times (t-\varepsilon,t])\subset U$. Furthermore, by the definition of $t$,
there is an $s\in(t-\varepsilon,t]$, such that $\widehat F|_{\s\times s}$ is continuous, so we may
assume (by adjusting $\delta$ if necessary) that $\widehat F\bigl((\s\cap[0,\delta))\times s\bigr)\subset U$.
Since $U$ is path-connected, it follows that $\widehat F(B)\subset U$. We claim that $\widehat F$ is
continuous on $0\times (t-\varepsilon,t+\varepsilon)$.

Take an $x\in (t-\varepsilon,t+\varepsilon)$ and let $\{x_i\}$ be a sequence in $B$ converging to $(0,x)$.
To show that $\{\widehat F(x_i)\}$ converges to $\widehat F(0,x)$ we must show that every open neighborhood 
$W$ of $\widehat F(0,x)$ contains all but finitely many elements of the sequence. Indeed, otherwise the elements
of $\{\widehat F(x_i)\}$ outside $W$ would have an accumulation point $u$ in the compact set $\overline U-W$.
This would imply that $F(0,x)$ and $\iota(u)$ are distinct accumulation points of the convergent sequence
$\{F(x_i)\}$, a contradiction.

Our argument shows that the assumption $t<1$ would lead to a contradiction, so we conclude that $t=1$,
and therefore $\widetilde F\colon\s\times I\to\hX$ is continuous.
\end{proof}

\begin{example}
The obvious map from $[0,1)$ to the Warsaw circle is a fibration for the class of 1-countable spaces, 
while the projection from the countable 
one-point union of intervals to $C\s$ is not.
\end{example}

If $X$ is first countable and Hausdorff then so is $\hX$: the Hausdorff property is obvious, for the other simply
observe that if we take a countable local basis around the point $x$ in $X$, then the path-components of the basic
sets containing 
the point $x$ constitute a countable local basis for $x$ in $\hX$. If in addition $\hX$ is locally compact 
(and therefore locally sequentially compact), then we have just proved that 
$\iota\colon \hX \to X$ is a fibration for the class of all first countable spaces. 
But a map between first countable spaces that has the 
covering homotopy property for maps between from first countable has 
automatically the covering homotopy property for maps  from arbitrary spaces, so we have the following

\begin{theorem}
\label{thmlifting}
Let $X$ be a first countable, Hausdorff space. If $\hX$ is locally compact then $\iota\colon\hX\to X$ is a lifting
space.
\end{theorem}
\begin{proof}
By \ref{hlpslem} and \ref{lclem}  we already know that $\iota$ has the covering homotopy property 
for 1-countable spaces. Since we assumed that $X$ is 1-countable that $\hX$ and $X^I$ are also 1-countable, and therefore
the space $(\hX\sqcap X^I)\times I$ is 1-countable, because it is a subspace of $\hX\times X^I$.
It follows that there exists the unique map $H$ that makes
commutative the following diagram
$$\xymatrix{
(\hX\sqcap X^I)\times 0 \ar[rr]^-{(\hat x,\alpha,0)\mapsto \hat x} \ar@{^(->}[d]& & \hX\ar[d]^\iota\\
(\hX\sqcap X^I)\times I \ar@{-->}[urr]^H\ar[rr]_-{(\hat x,\alpha,t)\mapsto \alpha(t)} & & X}
$$
Its adjoint map $\widetilde H\colon \hX\sqcap X^I\to \hX^I$ is the continuous inverse for the
continuous bijection $\bar\iota\colon \hX^I\to\hX\sqcap X^I$, therefore $\iota$ is a lifting space. 
\end{proof}

\begin{remark}
The above results can be easily extended to more general spaces. Indeed, let $\aleph$ be any cardinal
number, and let $\omega=\omega(\aleph)$ be the corresponding initial ordinal.
Then we may consider spaces for which every point has local bases of cardinality at most $\aleph$
and spaces that are $\omega$-compact, i.e. each 'sequence' indexed by $\omega$ has an accumulation 
point. Then we can repeat the above proofs word-by-word to show that if $X$ has local basis of 
cardinality at most $\aleph$ and if $\hX$ is locally $\omega$-compact then $\iota\hX\to X$ is 
a lifting space.
\end{remark}

To prove the converse of Theorem \ref{lclem} we need to assume that $X$ is a metric space.
Let us first show that if
$(X,d)$ is a path-connected metric space then $\hX$ is also metrizable. Indeed, a suitable metric on $\hX$ 
can be defined by taking into account the path structure on $X$.
Let the \emph{breadth} of a path $\alpha\colon I\to X$ be defined as 
$${\rm br}(\alpha):=\sup \big\{d\big(\alpha(0),\alpha(t)\big)+d\big(\alpha(t),\alpha(1)\big)\mid t\in [0,1]\big\}.$$
Then we get a metric $\rho$ on $\hX$ by
$$\rho(x,x') =\inf\big\{ {\rm br}(\alpha)\mid \alpha\colon (I,0,1)\to (X,x,x')\big\}.$$
Note that always $d(x,x')\le\rho(x,x')$, and that for a path-connected set $A\subseteq X$ we 
have $\diam_\rho(A)\le 2\cdot\diam_d(A).$ 

We can easily verify that the topology of $\hX$ is induced by $\rho$. 
In fact, let $C$ be a component of an open set $U\subseteq X$. For every $x\in C$ there 
is an $\varepsilon$-ball $B_d(x,\varepsilon)\subseteq U$, and since $d$-distance does not exceed 
$\rho$-distance, we also have that $B_\rho(x,\varepsilon)\subseteq U$. It follows that 
all points of $B_\rho(x,\varepsilon)$ can be  connected by 
a path in $U$, therefore $B_\rho(x,\varepsilon)\subseteq C$, and so $C$ is open with respect to the metric 
$\rho$. On the other hand, the ball $B_\rho(x,\epsilon)$
is clearly open with respect to the metric $d$, so the path-component of $B_\rho(x,\epsilon)$ 
containing $x$ is an open set in $\hX$ contained in that ball.

\begin{theorem}
\label{notcpt}
Suppose $X$ is a locally compact, path-connected metric space. If $\iota\colon\hX\to X$ 
has the homotopy lifting property for maps from $\s$ then $\hX$ is locally compact.
\end{theorem}
\begin{proof}
Assume by contradiction that $\hX$ is not locally compact, so that there is a point $x\in X$ that does not
possess any compact neighborhood in $\hX$. Let $\varepsilon_1>\varepsilon_2>\varepsilon_3>\ldots$ be a strictly 
decreasing sequence  converging to zero, such that all $B_d(x,\varepsilon_i)$ are relatively compact. 
For every $i$ the corresponding  ball $B_\rho(x,\varepsilon_i)\subset \hX$ 
is path-connected, contained in $B_d(x,\varepsilon_i)$ and, by the assumption, not relatively compact. Therefore, 
we can choose for each $i$ a sequence $x_1^i,x_2^i,\ldots$ 
of points in $B_\rho(x,\varepsilon_i)$ that converges in $X$ to a point in the closure of $B_d(x,\varepsilon_i)$ 
but does not have any accumulation points in $\hX$. 
For every $j$, we have $\rho(x^i_j,x^{i+1}_j)< \varepsilon_i+\varepsilon_{i+1}$, so we can find a path of breadth 
less then $\varepsilon_i+\varepsilon_{i+1}$ 
connecting $x^i_j$ and $x^{i+1}_j$. We can concatenate these paths to obtain a path $\alpha_j\colon I\to X$, running 
from $x^1_j$ to  $x^2_j$ on $[0,1/2]$, 
from to $x^2_j$ to $x^3_j$ on $[1/2,3/4]$, and so on. The paths $\alpha_j$ (taken in reverse direction) together with 
the constant path in $x$ define a homotopy 
$H\colon\s\times I\to X$, given by the formula
   $$H(u,t) = \begin{cases}
   x & \text{if } u=0  \\
   \alpha_{\frac{1}{u}}(1-t) & \text{otherwise}.
   \end{cases}$$ 
Since $H_0$ is a constant map it can be lifted to $\hX$, but we cannot lift the entire homotopy $H$, because 
then $\widehat H_1$ would send the convergent sequence $\s$ to the sequence
$x^1_1,x^1_2,x^1_3,\ldots$ that does not converge in $\hX$.
\end{proof}

We may now combine Theorem \ref{notcpt} with Theorems \ref{lclem} and \ref{thmlifting} to obtain the
following result.

\begin{theorem}
Let $X$ be a path-connected, locally compact metric space. Then $\iota\colon\hX\to X$ is a lifting
space if and only if $\hX$ is locally compact.
\end{theorem}

We do not know, whether the last result can be extended to all first-countable spaces.

Let $p\colon L\to X$ be a lifting space whose total space $L$ is locally path-connected. 
Then one can easily check that the induced map $\hat p\colon L\to\hX$ is also a lifting space.


\begin{thebibliography}{}
\bibitem{BDLM}
N. Brodskiy, J. Dydak, B. Labuz, A. Mitra, Covering maps for locally path-connected spaces. 
\emph{Fund. Math.} {\bf 218} (2012), no. 1, 13–46.

\bibitem{Cannon-Conner} J.~W.~Cannon, G.~R.~Conner, On the fundamental groups of 
one-dimensional spaces, \emph{Top. Appl.} {\bf 153} (2006) 2648-2672.

\bibitem{Conner-Herfort-Pavesic}
G. Conner, W. Herfort, P. Pave\v si\'c, Some anomalous examples of
lifting spaces, \emph{Top. Appl.} {\bf 239} (2018), 234-243.

\bibitem{Conner-Herfort-Pavesic2}
G. Conner, W. Herfort, P. Pave\v si\'c, Geometry of compact lifting spaces, to appear in
\emph{Monatshefte f\"ur Mathematik}.

\bibitem{Conner-Kent-Herfort-Pavesic}
G.~Conner, C.~Kent, W.~Herfort, P.~Pave\v si\'c, 
Uncountable groups and the geometry of inverse limits of coverings, submitted.

\bibitem{Fisher-Zastrow}
H. Fischer, A. Zastrow, Generalized universal covering spaces and the shape group. \emph{Fund. Math.} {\bf 197}
(2007), 167 – 196.

\bibitem{FRVZ}
H. Fischer, D. Repov\v s, \v Z. Virk, A. Zastrow, On semilocally simply connected spaces, \emph{Top. Appl.}
{\bf 158} (2011), 397-408.

\bibitem{Fox}
R.H. Fox, On shape, \emph{Fund. Math.}, {\bf 74} (1972), 47–71.

\bibitem{Hata85}
Masayoshi Hata.
\newblock On the structure of self-similar sets.
\newblock {\em Japan J. Appl. Math.}, 2(2):381--414, 1985.

\bibitem{KapovichKleiner00}
Michael Kapovich and Bruce Kleiner.
\newblock Hyperbolic groups with low-dimensional boundary.
\newblock {\em Ann. Sci. \'Ecole Norm. Sup. (4)}, 33(5):647--669, 2000.

\bibitem{Mandelbrot85}
Benoit~B. Mandelbrot.
\newblock On the dynamics of iterated maps. {VII}. {D}omain-filling
  (``{P}eano'') sequences of fractal {J}ulia sets, and an intuitive rationale
  for the {S}iegel discs.
\newblock In {\em Chaos, fractals, and dynamics ({G}uelph, {O}nt., 1981/1983)},
  volume~98 of {\em Lecture Notes in Pure and Appl. Math.}, pages 243--253.
  Dekker, New York, 1985.

\bibitem{Mardesic-Segal} 
S.~Marde\v si\'c, J.~Segal, \emph{Shape Theory}, North Holland Mathematical Library, Vol. 26 (North Holland
Publishing Company, Amsterdam, 1982).

\bibitem{Massopust89}
Peter~R. Massopust.
\newblock Fractal {P}eano curves.
\newblock {\em J. Geom.}, 34(1-2):127--138, 1989.

\bibitem{Pavesic-Piccinini}
P.~Pave\v si\'c, R.A.~Piccinini, \emph{Fibrations and their Classification}, Research and Exposition in
Mathematics, Vol. 33 (Heldermann Verlag, 2013)

\bibitem{Pavesic}
P.~Pave\v si\'c, \emph{Fibrations between mapping spaces}, \emph{Top. Appl.} {\bf 178} (2014), 276-287.


\bibitem{Shelah}
S.~Shelah, \emph{Can the fundamental (homotopy) group of a space be
the rationals?}, Proc. AMS, {\bf 103} (1988), 627-632.


\bibitem{Spanier} 
E.~H.~Spanier: \emph{Algebraic Topology}, (Springer-Verlag, New York 1966).
\end{thebibliography}
\end{document}